\documentclass[11pt, reqno]{amsart} \textwidth=14.9cm \oddsidemargin=1cm \evensidemargin=1cm
\usepackage{amsmath,amstext,amsbsy,amssymb,amscd}
\usepackage{amsmath}
\usepackage{amsxtra}
\usepackage{amscd}
\usepackage{amsthm}
\usepackage{amsfonts}
\usepackage{amssymb}
\usepackage{eucal}
\usepackage{color}
\usepackage{graphicx}%
\usepackage{bm}
\newtheorem{theorem}{Theorem}[section]
\newtheorem{lemma}[theorem]{Lemma}
\newtheorem{proposition}[theorem]{Proposition}

\theoremstyle{definition}
\newtheorem{definition}[theorem]{Definition}
\newtheorem{remark}[theorem]{Remark}

\definecolor{A}{rgb}{.75,1,.75}

\numberwithin{equation}{section}

\begin{document}

\title[Affine cellularity of affine Yokonuma-Hecke algebras]{Affine cellularity of affine Yokonuma-Hecke algebras}
\author[Weideng Cui]{Weideng Cui}
\address{School of Mathematics, Shandong University, Jinan, Shandong 250100, P.R. China.}
\email{cwdeng@amss.ac.cn}

\begin{abstract}
We establish an explicit algebra isomorphism between the affine Yokonuma-Hecke algebra $\widehat{Y}_{r,n}(q)$ and a direct sum of matrix algebras with coefficients in tensor products of affine Hecke algebras of type $A.$ As an application of this result, we show that $\widehat{Y}_{r,n}(q)$ is affine cellular in the sense of Koenig and Xi, and further prove that it has finite global dimension when the parameter $q$ is not a root of the Poincar\'{e} polynomial. As another application, we also recover the modular representation theory of $\widehat{Y}_{r,n}(q)$ previously obtained in [CW].
\end{abstract}



\thanks{\emph{\emph{2010} Mathematics Subject Classification}. 20C08, 16E10.}

\maketitle
\medskip
\section{Introduction}
\subsection{}
Inspired by Lusztig's work on the structure of cells and the base ring of affine Hecke algebras, Koenig and Xi [KX2] recently defined the notion of affine cellularity to generalize the notion of cellular algebras [GL] to algebras of not necessarily finite dimension over a noetherian domain $k.$ Extended affine Hecke algebras of type $A$ and affine Temperley-Lieb algebras were proved to be affine cellular in [KX2]. Further examples of affine cellular algebras include affine Hecke algebras of rank two with generic parameters [GM], KLR algebras of finite type [KL], BLN algebras ([C1] and [N]) and affine $q$-Schur algebras [C2].

\subsection{}
Yokonuma-Hecke algebras of general types were first introduced in the sixties by Yokonuma [Yo]. In the late 1990s and early 2000s, a new presentation of the Yokonuma-Hecke algebra $Y_{r,n}(q)$ of type $A$ was given by Juyumaya and Kannan [Ju1, JuK], which has been widely used for studying this algebra since then.

In recent years, many people are interested in the algebra $Y_{r,n}(q)$ from different perspectives. Some people studied $Y_{r,n}(q)$ in order to construct its associated knot invariant; see [Ju2], [JuL] and [ChL] and so on. Others are particularly interested in the representation theory of $Y_{r,n}(q).$ Chlouveraki and Poulain d'Andecy [ChPA1] gave explicit formulas for all irreducible representations of $Y_{r,n}(q)$ over $\mathbb{C}(q)$ and obtained a semisimplicity criterion for it. Moreover, they defined a new kind of algebras, called affine Yokonuma-Hecke algebras and denoted by $\widehat{Y}_{r,n}(q).$ In their subsequent paper [ChPA2], they studied the representation theory of the affine Yokonuma-Hecke algebra $\widehat{Y}_{r,n}(q)$ and the cyclotomic Yokonuma-Hecke algebra $Y_{r,n}^{d}(q)$. In particular, they gave the classification of irreducible representations of $Y_{r,n}^{d}(q)$ in the generic semisimple case.

Besides, Jacon and Poulain d'Andecy [JPA] (see also [ER]) gave an explicit algebraic isomorphism between the Yokonuma-Hecke algebra $Y_{r,n}(q)$ and a direct sum of matrix algebras over tensor products of Iwahori-Hecke algebras of type $A,$ which is in fact a special case of the results by G. Lusztig [L6, Section 34]. This allows them to give a description of the modular representation theory of $Y_{r,n}(q)$ and a complete classification of all Markov traces for it. Recently, Chlouveraki and S\'{e}cherre [ChS] proved that the affine Yokonuma-Hecke algebra is a particular case of the pro-$p$-Iwahori-Hecke algebras.

In [CW], we have established an equivalence between a module category of $\widehat{Y}_{r,n}(q)$ (resp. $Y_{r,n}^{d}(q)$) and its suitable counterpart for a direct sum of tensor products of affine Hecke algebras of type $A$ (resp. cyclotomic Hecke algebras), which allows us to give the classification of simple modules of $\widehat{Y}_{r,n}(q)$ and $Y_{r,n}^{d}(q)$ over an algebraically closed field of characteristic $p$ such that $p$ does not divide $r.$

\subsection{}
Since the affine Hecke algebra of type $A$ is affine cellular, it is natural to try to show that the affine Yokonuma-Hecke algebra $\widehat{Y}_{r,n}(q)$ is also affine cellular. In this paper, we will prove this fact by constructing an explicit algebra isomorphism between the affine Yokonuma-Hecke algebra $\widehat{Y}_{r,n}(q)$ and a direct sum of matrix algebras with coefficients in tensor products of various affine Hecke algebras of type $A$. As another application, we also recover the modular representation theory of $\widehat{Y}_{r,n}(q)$ previously obtained in [CW].

This paper is organized as follows. In Section 2, we recall Koenig and Xi's results on affine cellular algebras, and then review the axiomatic approach to studying them presented in [C2]. In Section 3, we give another presentation of the affine Yokonuma-Hecke algebra $\widehat{Y}_{r,n}(q)$. In Section 4, inspired by the work of Lusztig in [L6], we give the construction of an algebra $\widehat{E}_{r,n}$, which is in fact a direct sum of matrix algebras with coefficients in tensor products of extended affine Hecke algebras of type $A$. Moreover, we prove that $\widehat{E}_{r,n}$ satisfies these axiomatic properties $P_{1}$, $P_{2}$, $P_{3}$ and $P_{4}$ described in [C2, Section 3]. In Section 5, we follow Lusztig's approach in [L6] to establish an explicit algebra isomorphism between $\widehat{Y}_{r,n}(q)$ and $\widehat{E}_{r,n}$. Thus, we prove that $\widehat{Y}_{r,n}(q)$ is an affine cellular algebra. We further prove that it has finite global dimension when the parameter $q$ is not a root of the Poincar\'{e} polynomial.

{\it Additional remark}: Four months after the first version of this paper was made available on arXiv (arXiv:1510.02647), Poulain d'Andecy put up his preprint [PA] on arXiv to give another proof of the isomorphism theorem for affine Yokonuma-Hecke algebras.

\section{An axiomatic approach to affine cellular algebras}
\subsection{Affine cellular algebras}

In this subsection, we recall Koenig and Xi's ([KX2]) definition and results on affine cellular algebras. Throughout, we assume that $k$ is a Noetherian domain.

For two $k$-modules $V$ and $W,$ let $\tau$ be the switch map: $V\otimes W\rightarrow W\otimes V$ defined by $\tau(v\otimes w)=w\otimes v$ for $v\in V$ and $w\in W.$ A $k$-linear anti-automorphism $i$ of a $k$-algebra $A$ which satisfies $i^{2}=id_{A}$ will be called a $k$-\emph{involution} on $A.$ A commutative $k$-algebra $B$ is called an \emph{affine} $k$-\emph{algebra} if it is a quotient of a polynomial ring $k[x_{1},\ldots, x_{r}]$ in finitely many variables $x_{1},\ldots, x_{r}$ by some ideal $I.$

\begin{definition} (See [KX2, Definition 2.1].) Let $A$ be a unitary $k$-algebra with a $k$-involution $i$. A two-sided ideal $J$ in $A$ is called an \emph{affine cell ideal} if and only if the following data are given and the following conditions are satisfied:\vskip2mm
$(1)$ $i(J)=J.$

$(2)$ There exist a free $k$-module $V$ of finite rank and an affine $k$-algebra $B$ with a $k$-involution $\sigma$ such that $\Delta :=V\otimes_{k} B$ is an $A$-$B$-bimodule, where the right $B$-module structure is induced by the right regular $B$-module $B_{B}$.

$(3)$ There is an $A$-$A$-bimodule isomorphism $\alpha :J\rightarrow \Delta\otimes_{B}\Delta',$ where $\Delta':=B\otimes_{k}V$ is a $B$-$A$-bimodule with the left $B$-module induced by the left regular $B$-module ${}_{B}B$ and with the right $A$-module structure defined by $(b\otimes v)a :=\tau(i(a)(v\otimes b))$ for $a\in A$, $b\in B$ and $v\in V$, such that the following diagram is commutative:

\[\begin{CD}
J   @>\alpha>>\Delta\otimes_{B}\Delta'\\
@ViVV                  @VVu\otimes b\otimes_{B}b'\otimes v\mapsto v\otimes \sigma(b')\otimes_{B}\sigma(b)\otimes uV\\
J         @>\alpha>>   \Delta\otimes_{B}\Delta'.
\end{CD}\]~

The algebra $A$ together with the $k$-involution $i$ is called \emph{affine cellular} if and only if there is a $k$-module decomposition $A=J_{1}'\oplus J_{2}'\oplus\cdots \oplus J_{n}'$ (for some $n$) with $i(J_{l}')=J_{l}'$ for $1\leq l\leq n,$ such that, setting $J_{l} :=\bigoplus_{h=1}^{l}J_{h}',$ we have a chain of two-sided ideals of $A$: $$0=J_{0}\subset J_{1}\subset J_{2}\subset\cdots\subset J_{n}=A$$ so that each subquotient $J_{l}'=J_{l}/J_{l-1}$ ($1\leq l\leq n$) is an affine cell ideal of $A/J_{l-1}$ (with respect to the involution induced by $i$ on the quotient).

We call this chain an \emph{affine cell chain} for the affine cellular algebra $A.$ The subquotients of an affine cell chain will be called \emph{layers} of $A.$
\end{definition}

Given a free $k$-module $V$ of finite rank $n,$ an affine $k$-algebra $B$ and a $k$-bilinear form $\rho : V\otimes_{k} V\rightarrow B,$ we define an associative algebra $\mathbf{A}(V,B,\rho)$ as follows: $\mathbf{A}(V,B,\rho) :=V\otimes_{k} B\otimes_{k}V$ as a $k$-module, and the multiplication on $\mathbf{A}(V,B,\rho)$ is defined by $$(u_1\otimes b_1\otimes v_1)(u_2\otimes b_2\otimes v_2) :=u_1\otimes b_1\rho(v_1,u_2)b_2\otimes v_2$$ for all $u_1, u_2, v_1, v_2\in V$ and $b_1, b_2\in B.$ Moreover, if $B$ admits a $k$-involution $\sigma$ satisfying $\sigma\rho(v_1, v_2)=\rho(v_2, v_1),$ then $\mathbf{A}(V,B,\rho)$ admits a $k$-involution $\varrho$ which sends $u\otimes b \otimes v$ to $v\otimes \sigma(b)\otimes u$ for all $u,v\in V$ and $b\in B.$

An equivalent description of this construction is as follows: Given $V,$ $B,$ $\rho$ as above, we define the \emph{generalized matrix algebra} $(M_{n}(B), \rho)$ over $B$ with respect to $\rho$ in the following way. It equals the ordinary matrix algebra $M_{n}(B)$ of $n\times n$ matrices over $B$ as a $k$-space, but the multiplication is given by $$\widetilde{x}\cdot \widetilde{y}=x\Psi y$$ for all $x, y \in M_{n}(B),$ where $\widetilde{x}$ and $\widetilde{y}$ are elements of $(M_{n}(B), \rho)$ corresponding to $x$ and $y,$ respectively, and $\Psi$ is the matrix describing the bilinear form $\rho$ with respect to some basis of $V.$ Moreover, if $B$ admits a $k$-involution $\sigma$ satisfying $\sigma\rho(v_1, v_2)=\rho(v_2, v_1),$ then $(M_{n}(B), \rho)$ admits a $k$-involution $\kappa$ which sends $E_{jl}(b)$ to $E_{lj}(\sigma(b)),$ where $E_{jl}(b)$ denotes a square matrix whose $(j,l)$-entry is $b\in B$ and all the other entries are zero.

From the above discussion, we can easily get the following proposition about the description of affine cell ideals.
\begin{proposition}\label{kx2-proposition-affine-cell}
{\rm (See [KX2, Proposition 2.2].)} Let $k$ be a Noetherian domain, $A$ a unitary $k$-algebra with a $k$-involution $i$. A two-sided ideal $J$ in $A$ is an affine cell ideal if and only if $i(J)=J,$ and $J$ is isomorphic to some generalized matrix algebra $(M_{n}(B), \rho)$ for some affine $k$-algebra $B$ with a $k$-involution $\sigma,$ a free $k$-module $V$ of finite rank $n$ and a $k$-bilinear form $\rho :V\otimes_{k} V\rightarrow B,$ such that under this isomorphism, if a basis element $a$ of $J$ corresponds to $E_{jl}(b')$ for some $b'\in B,$ then $i(a)$ corresponds to $E_{lj}(\sigma(b'))$.
\end{proposition}

It has been claimed that the tensor product of two cellular algebras is also cellular ([KX1, p. 716]). In a similar way, we wil show that the tensor product of two affine cellular algebras is also affine cellular in the next lemma.
\begin{lemma}\label{lemma-tensor-products-aff}
If both algebras $(A, i)$ and $(B, j)$ are affine cellular over $k$, then the tensor product $A\otimes_{k} B$ is also affine cellular with respect to the $k$-involution $i\otimes j.$
\end{lemma}
\begin{proof}
As in Definition 2.1, we assume that there is a $k$-module decomposition $A=J_{1}'\oplus J_{2}'\oplus\cdots \oplus J_{n}'$ (for some $n$) (resp. $B=K_{1}'\oplus K_{2}'\oplus\cdots \oplus K_{m}'$ (for some $m$)) with $i(J_{r}')=J_{r}'$ for $1\leq r\leq n$ (resp. $j(K_{s}')=K_{s}'$ for $1\leq s\leq m$), such that, setting $J_{r} :=\bigoplus_{a=1}^{r}J_{a}'$ (resp. $K_{s} :=\bigoplus_{b=1}^{s}K_{b}'$), we have a chain of two-sided ideals of $A$ (resp. $B$): $$0=J_{0}\subset J_{1}\subset J_{2}\subset\cdots\subset J_{n}=A \text{ (resp. } 0=K_{0}\subset K_{1}\subset K_{2}\subset\cdots\subset K_{m}=B)$$ so that each $J_{r}'=J_{r}/J_{r-1}$ ($1\leq r\leq n$) (resp. $K_{s}'=K_{s}/K_{s-1}$ ($1\leq s\leq m$)) is an affine cell ideal of $A/J_{r-1}$ (resp. $B/K_{s-1}$).

We construct the following two-sided ideals of $A\otimes_{k} B$:
\begin{align*}
H_{am+b} :=J_{1}'\otimes_{k} B\oplus\cdots \oplus J_{a}'\otimes_{k} B\oplus J_{a+1}'\otimes_{k} K_{b}\quad \text{ for }0\leq a\leq n-1\text{ and }1\leq b\leq m.
\end{align*}

We leave it as an exercise to use Proposition \ref{kx2-proposition-affine-cell} to verify that the following chain of two-sided ideals of $A\otimes_{k} B$:
\begin{equation*}
0=H_{0}\subset H_{1}\subset H_{2}\subset\cdots\subset H_{m}\subset H_{m+1}\subset\cdots\subset H_{(n-1)m+1}\subset\cdots\subset H_{nm}=A\otimes_{k} B
\end{equation*}
is an affine cell chain of $A\otimes_{k} B$ with respect to the $k$-involution $i\otimes j,$ and that the set $\{J_{r}'\otimes_{k} K_{s}'\:|\:1\leq r\leq n\text{ and }1\leq s\leq m\}$ forms a complete set of layers in the affine cell chain. In fact, if each layer $J_{r}'\cong (M_{n_{1}}(B_1), \rho_1)$ for some affine $k$-algebra $B_1$, a free $k$-module $V_1$ of rank $n_1$ and a $k$-bilinear form $\rho_1 :V_1\otimes_{k} V_1\rightarrow B_1,$ and $K_{s}'\cong (M_{n_{2}}(B_2), \rho_2)$ for some affine $k$-algebra $B_2$, a free $k$-module $V_2$ of rank $n_2$ and a $k$-bilinear form $\rho_2 :V_2\otimes_{k} V_2\rightarrow B_2,$ then we have $J_{r}'\otimes_{k} K_{s}'\cong (M_{n_{1}n_{2}}(B_1\otimes_{k} B_2), \rho_1\otimes \rho_2),$ where $\rho_1\otimes \rho_2 :(V_1\otimes_{k} V_2)\otimes_{k} (V_1\otimes_{k} V_2)\rightarrow B_1\otimes_{k} B_2$ is given by $\rho_1\otimes \rho_2(v_1\otimes v_2, u_1\otimes u_2)=\rho_1(v_1, u_1)\otimes \rho_2(v_2, u_2).$

Thus, $A\otimes_{k} B$ together with the $k$-involution $i\otimes j$ is an affine cellular algebra.
\end{proof}

We will also use the following lemma.
\begin{lemma}\label{affine-cellular-algebra-field-extension}
{\rm (See [KX2, Lemma 2.4].)} Suppose that $K$ is another Noetherian domain and $\phi :k\rightarrow K$ is a homomorphism of rings with identity. If $A$ is an affine cellular $k$-algebra with an involution $i,$ then $K\otimes_{k} A$ is an affine cellular $K$-algebra with respect to the involution id$_{K}\otimes i.$
\end{lemma}

The following theorem plays an important role in investigating homological properties of affine cellular algebras.

\begin{theorem}\label{affine-cellua-YH-alg}
{\rm (See [KX2, Theorem 4.4].)} Let $A$ be an affine cellular algebra with an affine cell chain $0=J_{0}\subset J_{1}\subset J_{2}\subset\cdots\subset J_{n}=A$ such that $J_{l}/J_{l-1}\cong (M_{n_{l}}(B_{l}), \rho_{l})$ as in Proposition \ref{kx2-proposition-affine-cell}. Here we call the algebra $\bigoplus_{l=1}^{n}M_{n_{l}}(B_{l})$ the asymptotic algebra of $A.$

Suppose that each $B_{l}$ satisfies $\mathrm{rad}(B_{l})=0.$ Moreover, suppose that each layer $J_{l}/J_{l-1}$ is idempotent and contains a non-zero idempotent element in $A/J_{l-1}.$ Then we have\vskip1mm

$(1)$ The parameter set of simple $A$-modules equals the parameter set of simple modules of the asymptotic algebra, that is, a finite union of affine spaces $($one for each $B_{l}$$)$.

$(2)$ The unbounded derived module category $D(A$-$\mathrm{Mod})$ of $A$ admits a stratification, that is, an iterated recollement whose strata are the derived module categories of the various affine $k$-algebras $B_{l}.$

$(3)$ The global dimension $\mathrm{gldim}(A)$ is finite if and only if $\mathrm{gldim}(B_{l})$ is finite for all $l.$
\end{theorem}

\begin{remark} Koenig [Koe, p. 531] called an affine cellular algebra with the assumptions stated as in Theorem \ref{affine-cellua-YH-alg} an affine quasi-hereditary algebra, since it implies the crucial homological properties analogous to known results about quasi-hereditary algebras and highest weight categories; see also [Kle] for the graded version of affine quasi-heredity.
\end{remark}

\subsection{An axiomatic approach}

In this subsection, we shall review an \emph{axiomatic approach} to affine cellular algebras established in [C2]. We [C2] show which conditions in [L4, Section 1] are required to prove that an algebra with a cell theory is affine cellular in the sense of Koenig and Xi.

Let $k=\mathbb{Z}[v, v^{-1}]$ $($$v$ an indeterminate$)$, and let $A$ be an associative $k$-algebra with a $k$-involution $i.$ Given a set $X,$ we say that an embedding $X\rightarrow A$ ($\lambda\mapsto 1_{\lambda}$) is a \emph{generalized unit} for $A$ if $1_{\lambda}1_{\lambda'}=\delta_{\lambda, \lambda'}1_{\lambda}$ for $\lambda, \lambda'\in X$ and $A=\sum_{\lambda, \lambda'\in X}1_{\lambda}A1_{\lambda'}.$

We assume that we are given a basis $B$ of $A$ as a $k$-module which is \emph{compatible} with the generalized unit in the following sense: the elements $1_{\lambda}$ ($\lambda\in X$) lie in $B$ and any $b\in B$ is contained in $1_{\lambda}A1_{\lambda'}$ for some uniquely determined $\lambda, \lambda'\in X.$

We assume that the structure constants $c_{b, b'}^{b''}\in k$ are given by
\begin{align*}
bb'=\sum_{b''\in B}c_{b, b'}^{b''}b''\quad\text{for }b, b'\in B.
\end{align*}

We say that a two-sided ideal of $A$ is \emph{based} if it is the span of a subset of $B.$ For $b, b'\in B,$ we say that $b\preceq_{LR} b'$ if $b$ lies in every based two-sided ideal which contains $b'.$ We say that $b\sim_{LR} b'$ if $b\preceq_{LR} b'$ and $b'\preceq_{LR} b.$ The equivalence classes for $\sim_{LR}$ are called 2-\emph{cells}.

For each 2-cell $\mathbf{c}$ in $B,$ let $A_{\mathbf{c}}$ be the $k$-subspace of $A$ spanned by $\mathbf{c}.$ There is an associative algebra structure on $A_{\mathbf{c}}$ in which the product of $b, b'\in \mathbf{c}$ is equal to $\sum_{b''\in \mathbf{c}}c_{b, b'}^{b''}b''.$ Note that the algebra $A_{\mathbf{c}}$ is naturally a \emph{subquotient} of $A$.

In the following we define a function $a: B\rightarrow \mathbb{N}\cup \{\infty\}.$ Let $\mathbf{c}$ be a 2-cell and let $L$ be the $\mathbb{Z}[v^{-1}]$-submodule of $A_{\mathbf{c}}$ generated by $\mathbf{c}.$

Let $b\in \mathbf{c}.$ If there is $n\in \mathbb{Z}_{\geq 0}$ such that $v^{-n}bL\subset L,$ then we define $a(b)$ to be the smallest such $n.$ If there is no such $n,$ we set $a(b)=\infty.$

We say that $B$ has \emph{property} $P_{1}$ if\vskip1mm

(a) the number of 2-cells in $B$ is countable, and is indexed by a partially ordered \hspace*{3.05em}set; further, the partial order $\preceq$ is compatible with the partial order $\preceq_{LR}$ and \hspace*{3.05em}for a given 2-cell $\mathbf{c},$ the set $\{\mathbf{c}'\:|\:\mathbf{c}'\preceq \mathbf{c}\}$ is finite.

(b) $a(b)$ is finite for all $b\in B.$

(c) For each 2-cell $\mathbf{c}$ and each $\lambda_{1}\in X,$ the restriction of $a$ to $\mathbf{c}1_{\lambda_{1}}$ is constant.\vskip1mm

Assume that $B$ has property $P_{1}.$ We can define for each 2-cell $\mathbf{c}$ the \emph{asymptotic ring} $A_{\mathbf{c}}^{\infty}$ with $\mathbb{Z}$-basis $t_{\mathbf{c}}=\{t_{b}\:|\:b\in \mathbf{c}\},$ and multiplication defined by $$t_{b}t_{b'}=\sum_{b''\in \mathbf{c}}\gamma_{b, b'}^{b''}t_{b''},$$
where $\gamma_{b, b'}^{b''}\in \mathbb{Z}$ is given by $v^{-a(b)}c_{b, b'}^{b''}=\gamma_{b, b'}^{b''}~\mathrm{mod}~v^{-1}\mathbb{Z}[v^{-1}].$

We say that $B,$ equipped with property $P_{1},$ has \emph{property} $P_{2}$ if for any 2-cell $\mathbf{c},$ the $\mathbb{Z}$-algebra $A_{\mathbf{c}}^{\infty}$ admits a generalized unit $\mathcal{D}_{\mathbf{c}}\rightarrow A_{\mathbf{c}}^{\infty},$ where $\mathcal{D}_{\mathbf{c}}$ is a finite set, and the basis $t_{\mathbf{c}}$ is compatible with this generalized unit.

We will identify $\mathcal{D}_{\mathbf{c}}$ with a subset of $\mathbf{c},$ so that the embedding $\mathcal{D}_{\mathbf{c}}\rightarrow t_{\mathbf{c}}$ is $d\mapsto t_{d}.$ In this case, the asymptotic ring $A_{\mathbf{c}}^{\infty}$ has 1, namely $1=\sum_{d\in \mathcal{D}_{\mathbf{c}}}t_{d}.$

Assume that $B$ has property $P_{1}.$ We say that $B$ has \emph{property} $P_{3}$ if we have the following equations for any $b_{1}, b_{2}, b_{3}, b'\in B$ with $b', b_{2}$ belonging to the same 2-cell $\mathbf{c}$:
$$\sum\limits_{b\in \mathbf{c}} c_{b_{1}, b_{2}}^{b}\gamma_{b, b_{3}}^{b'}=\sum\limits_{b\in \mathbf{c}} c_{b_{1}, b}^{b'}\gamma_{b_{2}, b_{3}}^{b},$$
$$\sum\limits_{b\in \mathbf{c}} \gamma_{b_{1}, b_{2}}^{b}c_{b, b_{3}}^{b'}=\sum\limits_{b\in \mathbf{c}} \gamma_{b_{1}, b}^{b'}c_{b_{2}, b_{3}}^{b}.$$

Let $\mathbf{c}$ be a 2-cell. Assume that $G_{\mathbf{c}}$ is a reductive group over $\mathbb{C}$ and that $\mathrm{Irr}~ G_{\mathbf{c}}$ is the set of irreducible complex representations of $G_{\mathbf{c}}.$ Let $T_{\mathbf{c}}$ be the set of triples $(d, d', s),$ where $d, d'\in \mathcal{D}_{\mathbf{c}}$ and $s\in \mathrm{Irr}~ G_{\mathbf{c}}.$ Let $J_{\mathbf{c}}$ be the free abelian group on $T_{\mathbf{c}}$ with a ring structure defined by $$(d_{1},d_{1}',s)(d_{2},d_{2}',s')=\delta_{d_{1}',d_{2}}\sum\limits_{s''\in \mathrm{Irr}~ G_{\mathbf{c}}}c_{s,s'}^{s''}(d_{1},d_{2}',s''),$$ where $c_{s,s'}^{s''}$ is the multiplicity of $s''$ in the tensor product $s\otimes s'$.

Since $\mathcal{D}_{\mathbf{c}}$ is a finite set, we will use $\{1, 2,\ldots, n_{\mathbf{c}}\}$ to label these elements in it, where $n_{\mathbf{c}}=|\mathcal{D}_{\mathbf{c}}|$. From now on we will always use this fixed label.

We say that $B,$ endowed with properties $P_{1}$ and $P_{2},$ has \emph{property} $P_{4}$ if the following conditions hold.\vskip1mm

(a) For each 2-cell $\mathbf{c},$ there is a bijection between $\mathbf{c}$ and the set $C=\{(j, l, s)\:|\:1\leq j,l\leq  n_{\mathbf{c}}, s\in \mathrm{Irr}~ G_{\mathbf{c}}\}.$ Moreover, if $b\in \mathbf{c}$ corresponds to $(j, l, s),$ we have $i(b)$ corresponds to $(l, j, \sigma(s)),$ where $\sigma(s)$ denotes the dual representation of $s.$

Hereafter, we will identify $\mathbf{c}$ with the set $C.$

(b) There exists a ring isomorphism $A_{\mathbf{c}}^{\infty}\rightarrow J_{\mathbf{c}},$ that is, the asymptotic ring $A_{\mathbf{c}}^{\infty}$ is isomorphic to an $n_{\mathbf{c}}\times n_{\mathbf{c}}$ matrix algebra over $B_{\mathbf{c}},$ where $B_{\mathbf{c}}$ is the representation ring of $G_{\mathbf{c}}.$ The isomorphism is given by $t_{b}\mapsto E_{jl}(s)$ for $b=(j, l, s)\in \mathbf{c}.$\vskip1mm

In [C2, Theorem 3.4] we have proved the following result by using Proposition \ref{kx2-proposition-affine-cell}.
\begin{theorem}\label{affine-cell-theorem}
Let $k=\mathbb{Z}[v, v^{-1}]$ $($$v$ an indeterminate$)$. If a $k$-basis $B$ of $A$ with a $k$-involution $i$ satisfies properties $P_{1},$ $P_{2},$ $P_{3}$ and $P_{4},$ then $A$ is an affine cellular algebra with respect to $i.$ Moreover, when $\mathbf{c}$ runs through all 2-cells in $B,$ the set consisting of the algebras $A_{\mathbf{c}}$ forms a complete set of layers in an affine cell chain of $A.$
\end{theorem}

\section{Another presentation of affine Yokonuma-Hecke algebras}

In this section, we shall give a new presentation of the affine Yokonuma-Hecke algebra $\widehat{Y}_{r,n}(q)$. Let us first recall the definition of the Yokonuma-Hecke algebra $Y_{r,n}(q)$ and give another presentation of it.

\subsection{Yokonuma-Hecke algebras}
Let $r, n\in \mathbb{N},$ $r, n\geq1,$ and let $\zeta=e^{2\pi i/r}.$ Let $q$ be an indeterminate and let $\mathcal{R}=\mathbb{Z}[\frac{1}{r}][q,q^{-1},\zeta].$ By [ChPA1, $\S2.1$] the \emph{Yokonuma-Hecke algebra} $Y_{r,n}=Y_{r,n}(q)$ is an $\mathcal{R}$-associative algebra generated by the elements $t_{1},\ldots,t_{n}$, $h_{1},\ldots,h_{n-1}$ satisfying the following relations:
\begin{equation}\label{rel-def-Y1-ac}\begin{array}{rclcl}
h_ih_j\hspace*{-7pt}&=&\hspace*{-7pt}h_jh_i && \mbox{for $1\leq i,j\leq n-1$ such that $\vert i-j\vert \geq 2$,}\\[0.1em]
h_ih_{i+1}h_i\hspace*{-7pt}&=&\hspace*{-7pt}h_{i+1}h_ih_{i+1} && \mbox{for $1\leq i\leq n-2$,}\\[0.1em]
t_it_j\hspace*{-7pt}&=&\hspace*{-7pt}t_jt_i &&  \mbox{for $1\leq i,j\leq n$,}\\[0.1em]
h_it_j\hspace*{-7pt}&=&\hspace*{-7pt}t_{s_i(j)}h_i && \mbox{for $1\leq i\leq n-1$ and $1\leq j\leq n$,}\\[0.1em]
t_i^r\hspace*{-7pt}&=&\hspace*{-7pt}1 && \mbox{for $1\leq i\leq n$,}\\[0.2em]
h_{i}^{2}\hspace*{-7pt}&=&\hspace*{-7pt}1+(q-q^{-1})e_{i}h_{i} && \mbox{for $1\leq i\leq n-1$,}
\end{array}
\end{equation}
where $s_{i}$ is the transposition $(i,i+1)$ in the symmetric group $\mathfrak{S}_n$ on $n$ letters, and for each $1\leq i\leq n-1$,
$$e_{i} :=\frac{1}{r}\sum\limits_{s=0}^{r-1}t_{i}^{s}t_{i+1}^{-s}.$$

Note that the elements $e_{i}$ are idempotents in $Y_{r,n}.$ The elements $h_{i}$ are invertible with the inverse given by
\begin{equation*}\label{inverse-ac}
h_{i}^{-1}=h_{i}-(q-q^{-1})e_{i}\quad\mbox{for}~1\leq i\leq n-1.
\end{equation*}

Let $w\in \mathfrak{S}_{n},$ and let $w=s_{i_1}\cdots s_{i_{r}}$ be a \emph{reduced expression} of $w.$ By Matsumoto's lemma, the element $h_{w} :=h_{i_1}h_{i_2}\cdots h_{i_{r}}$ does not depend on the choice of the reduced expression of $w$. Let $\ell$ denote the length function on $\mathfrak{S}_{n}.$ Then we have
\begin{align}\label{multi-formula-ac}
h_{s_{i}}h_{w}=\begin{cases}h_{s_{i}w}& \hbox {if } \ell(s_{i}w)>\ell(w), \\h_{s_{i}w}+(q-q^{-1})e_{i}h_{w}& \hbox {if } \ell(s_{i}w)<\ell(w).\end{cases}
\end{align}

Using the multiplication formula \eqref{multi-formula-ac}, Juyumaya [Ju] has proved that the following set is an $\mathcal{R}$-basis for $Y_{r,n}$:
\begin{equation}\label{basis-ac}
\mathcal{B}_{r,n}=\{t_{1}^{k_1}\cdots t_{n}^{k_n}h_{w}\:|\:0\leq k_1,\ldots,k_{n}\leq r-1\text{ and }w\in \mathfrak{S}_{n}\}.
\end{equation}
Thus, $Y_{r,n}$ is a free $\mathcal{R}$-module of rank $r^{n}n!.$

\subsection{Combinatorics of multipartitions and multitableaux}
In this subsection we recall some combinatorial objects which we will use. $\lambda=(\lambda_{1},\ldots,\lambda_{k})$ is called a \emph{partition} of $n$ if it is a finite sequence of weakly decreasing nonnegative integers whose sum is $n.$ We write $\lambda\vdash n$ if $\lambda$ is a partition of $n,$ and we define $|\lambda| :=n.$

To a partition $\lambda$ we associate a \emph{Young diagram}, which is the set $$[\lambda] :=\{(i,j)\:|\:i\geq 1~\mathrm{and}~1\leq j\leq \lambda_{i}\}.$$ We shall regard $[\lambda]$ as a left-justified array of boxes such that there exist $\lambda_{j}$ boxes in the $j$-th row for $j=1,\ldots,k.$ We write $\theta=(a,b)\in [\lambda]$ if the box $\theta$ lies in row $a$ and column $b$ of $[\lambda].$

For $\lambda\vdash n,$ we define a $\lambda$-\emph{tableau} (or \emph{tableau of shape} $\lambda$) by replacing each node of $[\lambda]$ with one of the integers $1,2,\ldots,n,$ allowing no repeats. For $\lambda\vdash n,$ we say that a $\lambda$-tableau $\mathfrak{t}$ is \emph{standard} if the entries in each row (resp. column) of $\mathfrak{t}$ \emph{increase} from left to right (resp. from top to bottom).

The combinatorial objects appearing in the representation theory of the Yokonuma-Hecke algebra $Y_{r,n}$ will be $r$-multipartitions. By definition, an $r$-\emph{multipartition} of $n$ is an ordered $r$-tuple $\bm{\lambda}=(\lambda^{(1)},\lambda^{(2)},\ldots,\lambda^{(r)})$ of partitions $\lambda^{(k)}$ such that $\sum_{k=1}^{r}|\lambda^{(k)}|=n.$ We call $\lambda^{(k)}$ the $k$-th \emph{component} of $\bm{\lambda}.$ We denote by $\mathcal{P}_{r,n}$ the set of $r$-multipartitions of $n.$ The \emph{Young diagram} $[\bm{\lambda}]$ of an $r$-multipartition $\bm{\lambda}$ is an ordered $r$-tuple of the Young diagram of its components. We write $\bm{\theta}=(\theta, s)\in [\bm{\lambda}]$ if the box $\theta$ lies in the Young diagram of the component $\lambda^{(s)}$ of $\bm{\lambda}.$

Assume that $\bm{\lambda}\in \mathcal{P}_{r,n}.$ A $\bm{\lambda}$-\emph{multitableau} (or \emph{multitableau of shape} $\bm{\lambda}$) $\mathfrak{t}=(\mathfrak{t}^{(1)},\ldots,\mathfrak{t}^{(r)})$ is defined by replacing each node of $[\bm{\lambda}]$ with one of the integers $1,2,\ldots,n,$ allowing no repeats. We will call the number $n$ the \emph{size} of $\mathfrak{t}$ and $\mathfrak{t}^{(k)}$ the $k$-th \emph{component} of $\mathfrak{t}.$

For each $\bm{\lambda}\in \mathcal{P}_{r,n},$ a $\bm{\lambda}$-multitableau $\mathfrak{t}$ is called \emph{standard} if the numbers \emph{increase} along any row (from left to right) and down any column (from top to bottom) of each diagram in $[\bm{\lambda}].$ For $\bm{\lambda}\in \mathcal{P}_{r,n},$ let $\mathrm{Std}(\bm{\lambda})$ denote the set of standard $\bm{\lambda}$-multitableaux of size $n$.

Let $\bm{\lambda}\in \mathcal{P}_{r,n}$ and $\mathfrak{t}$ be a $\bm{\lambda}$-multitableau. For $j=1,\ldots,n$, we define $\mathrm{p}_{\mathfrak{t}}(j)=k$ if $j$ appears in the $k$-th component $\mathfrak{t}^{(k)}$ of $\mathfrak{t}.$

\subsection{Another presentation of the Yokonuma-Hecke algebra}
We fix once and for all a total order on the set of $r$-th roots of unity via setting $\zeta_{k} :=\zeta^{k-1}$ for $1\leq k\leq r.$ We denote by $\underline{\mathfrak{s}}_{n}$ the set of $n$-tuples of $r$-th roots of unity. There is a natural action of the symmetric group $\mathfrak{S}_n$ on $\underline{\mathfrak{s}}_{n}$ by permuting $n$-tuples.

Let $\mathcal{A}=\mathbb{Z}[q, q^{-1}]$ $($$q$ an indeterminate$)$. We define $H_{r,n}$ to be an associative $\mathcal{A}$-algebra generated by the elements $\{g_{i}\:|\:i=1,\ldots,n-1\}$ and $\{1_{\lambda}\:|\:\lambda\in \underline{\mathfrak{s}}_{n}\}$ with the following relations:
\begin{equation}\label{rel-def-Y1-ac-another}\begin{array}{rclcl}
g_ig_j\hspace*{-7pt}&=&\hspace*{-7pt}g_jg_i && \mbox{for $1\leq i,j\leq n-1$ such that $\vert i-j\vert \geq 2$,}\\[0.1em]
g_ig_{i+1}g_i\hspace*{-7pt}&=&\hspace*{-7pt}g_{i+1}g_ig_{i+1} && \mbox{for $1\leq i\leq n-2$,}\\[0.1em]
g_{i}1_{\lambda}\hspace*{-7pt}&=&\hspace*{-7pt}1_{s_{i}(\lambda)}g_{i} &&  \mbox{for $1\leq i\leq n-1$ and $\lambda\in \underline{\mathfrak{s}}_{n}$,}\\[0.1em]
\sum_{\lambda\in \underline{\mathfrak{s}}_{n}}1_{\lambda}\hspace*{-7pt}&=&\hspace*{-7pt}1,\\[0.1em]
1_\lambda1_{\lambda'}\hspace*{-7pt}&=&\hspace*{-7pt}\delta_{\lambda, \lambda'}1_{\lambda}&&\mbox{for all $\lambda, \lambda'\in \underline{\mathfrak{s}}_{n},$}\\[0.1em]
g_{i}^{2}\hspace*{-7pt}&=&\hspace*{-7pt}1+(q-q^{-1})\sum\limits_{\lambda\mid s_{i}\in W_{\lambda}}g_{i}1_{\lambda}&& \mbox{for $1\leq i\leq n-1,$}
\end{array}
\end{equation}
where $W_{\lambda}=\{w\in \mathfrak{S}_{n}\:|\:w(\lambda)=\lambda\}.$ Note that the elements $g_{i}$ are invertible. By Matsumoto's lemma, we see that for each $w\in \mathfrak{S}_{n},$ the element $g_{w} :=g_{i_1}g_{i_2}\cdots g_{i_{r}}$ is well-defined if $w=s_{i_1}\cdots s_{i_{r}}$ is a reduced expression of $w.$

We denote by $\mathcal{P}_{r,n}^{1}$ the set of elements of $\mathcal{P}_{r,n}$ of the form $((1^{n_{1}}), (1^{n_{2}}),\ldots, (1^{n_{r}})),$ where $n_{i}\geq 0$ and $\sum_{i}n_{i}=n.$ We define
\begin{equation*}\label{mn-ac}
\mathrm{Std}^{1} :=\{\mathfrak{s}\:|\:\mathfrak{s}\in \mathrm{Std}(\bm{\lambda})\text{ for }\bm{\lambda}\in \mathcal{P}_{r,n}^{1}\}.
\end{equation*}

The following lemma appears in the proof of Proposition 1 in the first version of [ER], but it has been deleted in the current version of this paper. For convenience, we give Espinoza and Ryom-Hansen's proof here.
\begin{lemma}\label{yokonuma-hecke-bijection}
There exists a natural bijection between the sets $\underline{\mathfrak{s}}_{n}$ and $\mathrm{Std}^{1}.$
\end{lemma}
\begin{proof}
Assume that $(\zeta_{i_1},\ldots,\zeta_{i_{n}})$ is an $n$-tuple of $r$-th roots of unity. We define the corresponding multitableau $\mathfrak{s}$ as follows:\vskip1mm

1. We put the number $1$ in the box $((1,1),i_{1})$ of $\mathfrak{s}.$

2. If $i_{2}=i_{1},$ we then put the number $2$ in the box $((2,1),i_{1});$ otherwise, we put the \hspace*{0.94cm}number $2$ in the box $((1,1),i_{2}).$

3. Assume that $i_3=i_2.$ If in case 2 we have $i_2=i_1,$ then we put the number $3$ in \hspace*{0.95cm}the box $((3,1),i_{1});$ if in case 2 we have $i_2\neq i_1,$ then we put the number $3$ in the box \hspace*{0.94cm}$((2,1),i_{2}).$ Assume that $i_3\neq i_2.$ Then we put the number $3$ in the box $((2,1),i_{1})$ \hspace*{0.95cm}if $i_3=i_1,$ or in the box $((1,1),i_{3})$ if $i_3\neq i_1.$\vskip1mm

Continuing in this way we get a standard multitableau $\mathfrak{s}$ of shape $((1^{c_{1}}), (1^{c_{2}}),\ldots, (1^{c_{r}})),$ where $c_j$ is the number of times that $j$ appears in the set $\{i_1,\ldots,i_{n}\}.$ Conversely, for each standard multitableau $\mathfrak{s}$ of shape $((1^{c_{1}}), (1^{c_{2}}),\ldots, (1^{c_{r}})),$ we have an associated $n$-tuple of $r$-th roots of unity $(\zeta_{\mathrm{p}_{\mathfrak{s}}(1)},\ldots,\zeta_{\mathrm{p}_{\mathfrak{s}}(n)}).$ It is easy to see that this defines the inverse of the first map, and thus we have obtained the desired bijection.
\end{proof}

From the proof of Lemma \ref{yokonuma-hecke-bijection}, we can see that if an $n$-tuple of $r$-th roots of unity $(\zeta_{k_1},\ldots,\zeta_{k_{n}})$ corresponds to a standard multitableau $\mathfrak{s}$ under the bijection, we have $k_{i}=k_{i+1}$ if and only if $i$ and $i+1$ belong to the same column of some component of $\mathfrak{s}.$ By this and [ER, Propositions 31 and 33] we can get the following result, which gives another presentation of $Y_{r,n}$.
\begin{lemma}\label{isom-YeckeHeck-alge-ac}
Let $H_{r,n}^{\mathcal{R}}=\mathcal{R}\otimes_{\mathcal{A}} H_{r,n}.$ Then we have an $\mathcal{R}$-algebra isomorphism $Y_{r,n}\overset{\sim}{\longrightarrow}H_{r,n}^{\mathcal{R}}.$ Moreover, the following set
\begin{equation}\label{pbwbasis-ac-hrn-add-add}
\big\{1_{\lambda}g_{w}\:|\:\lambda\in \underline{\mathfrak{s}}_{n} \text{ and }w\in \mathfrak{S}_{n}\big\}
\end{equation}
forms an $\mathcal{A}$-basis of $H_{r,n}.$
\end{lemma}

\subsection{The affine Yokonuma-Hecke algebra and its new presentation}
By [ChPA1, $\S3.1$] the \emph{affine Yokonuma-Hecke algebra} $\widehat{Y}_{r,n}=\widehat{Y}_{r,n}(q)$ is an $\mathcal{R}$-associative algebra generated by the elements $t_{1},\ldots,t_{n}$, $h_{1},\ldots,h_{n-1}$, $Y_{1}^{\pm1},$ in which the generators $t_{1},\ldots,t_{n}$, $h_{1},$ $\ldots,h_{n-1}$ satisfy the same relations as in \eqref{rel-def-Y1-ac}, together with the following relations involving the generators $Y_{1}^{\pm1}$:
\begin{equation}\label{rel-def-Y2-ac}\begin{array}{rclcl}
Y_{1}Y_{1}^{-1}\hspace*{-7pt}&=&\hspace*{-7pt}Y_{1}^{-1}Y_{1}=1,\\[0.1em]
h_{1}Y_{1}h_{1}Y_{1}\hspace*{-7pt}&=&\hspace*{-7pt}Y_{1}h_{1}Y_{1}h_{1}, \\[0.1em]
h_{i}Y_{1}\hspace*{-7pt}&=&\hspace*{-7pt}Y_{1}h_{i} &&  \mbox{for $2\leq i\leq n-1$,}\\[0.1em]
t_{j}Y_{1}\hspace*{-7pt}&=&\hspace*{-7pt}Y_{1}t_{j} && \mbox{for $1\leq j\leq n$.}
\end{array}
\end{equation}

We now define $\widehat{H}_{r,n}$ to be an associative $\mathcal{A}$-algebra with generators $\{g_{i}\:|\:i=1,\ldots,n-1\}$, $\{1_{\lambda}\:|\:\lambda\in \underline{\mathfrak{s}}_{n}\}$ and $X_{1}^{\pm1}$, in which the generators $g_{i}$ $(1\leq i\leq n-1)$ and $1_{\lambda}$ $(\lambda\in \underline{\mathfrak{s}}_{n})$ satisfy the same relations as in \eqref{rel-def-Y1-ac-another}, together with the following relations involving the generators $X_{1}^{\pm1}$:
\begin{equation}\label{rel-def-Y2-ac-another}\begin{array}{rclcl}
X_{1}X_{1}^{-1}\hspace*{-7pt}&=&\hspace*{-7pt}X_{1}^{-1}X_{1}=1,\\[0.1em]
g_{1}X_{1}g_{1}X_{1}\hspace*{-7pt}&=&\hspace*{-7pt}X_{1}g_{1}X_{1}g_{1}, \\[0.1em]
g_{i}X_{1}\hspace*{-7pt}&=&\hspace*{-7pt}X_{1}g_{i} &&  \mbox{for $2\leq i\leq n-1$,}\\[0.1em]
X_{1}1_{\lambda}\hspace*{-7pt}&=&\hspace*{-7pt}1_{\lambda}X_{1}&&\mbox{for all $\lambda\in \underline{\mathfrak{s}}_{n}.$}
\end{array}
\end{equation}

By using Matsumoto's lemma again, we see that for each $w\in \mathfrak{S}_{n},$ both elements $h_{w} :=h_{i_1}h_{i_2}\cdots h_{i_{r}}\in \widehat{Y}_{r,n}$ and $g_{w} :=g_{i_1}g_{i_2}\cdots g_{i_{r}}\in \widehat{H}_{r,n}$ are well-defined if $w=s_{i_1}\cdots s_{i_{r}}$ is a reduced expression of $w.$ Moreover, both elements $h_{w}$ and $g_{w}$ are invertible.

We define inductively elements $Y_{2},\ldots,Y_{n}$ in $\widehat{Y}_{r,n}$ by
\begin{equation}\label{X2n}
Y_{i+1} :=h_{i}Y_{i}h_{i}\quad\text{ for }1\leq i\leq n-1.
\end{equation}
By [ChPA1, Proposition 1 and Lemma 1] we have
\begin{equation*}
Y_{i}Y_{j}=Y_{j}Y_{i}\quad\text{and}\quad Y_{j}t_{i}=t_{i}Y_{j}\quad \text{for }1\leq i, j\leq n,
\end{equation*}
and for any $1\leq i\leq n-1$ we have
\begin{equation*}
h_{i}Y_{j}=Y_{j}h_{i}\quad\mathrm{for}~1\leq j\leq n~\mathrm{such~that}~j\neq i, i+1.\label{giXj}
\end{equation*}

In a similar way, we define inductively elements $X_{2},\ldots,X_{n}$ in $\widehat{H}_{r,n}$ by
\begin{equation}\label{X2n-add-add-addadd}
X_{i+1} :=g_{i}X_{i}g_{i}\quad\text{ for }1\leq i\leq n-1.
\end{equation}
It can be easily proved that
\begin{equation*}
X_{i}X_{j}=X_{j}X_{i}\quad\text{and}\quad X_{j}1_{\lambda}=1_{\lambda}X_{j} \quad \text{for }1\leq i, j\leq n \text{ and } \lambda\in \underline{\mathfrak{s}}_{n}.
\end{equation*}
Moreover, for any $1\leq i\leq n-1$ we have
\begin{equation*}
g_{i}X_{j}=X_{j}g_{i}\quad\mathrm{for}~1\leq j\leq n~\mathrm{such~that}~j\neq i, i+1.\label{giXj}
\end{equation*}

By [ChPA2, Theorem 4.4] (see also [CW, Theorem 2.5]), we have the following PBW basis for $\widehat{Y}_{r,n}$:
\begin{equation}\label{pbwbasis-ac}
\big\{Y_{1}^{\alpha_{1}}\cdots Y_{n}^{\alpha_{n}}t_{1}^{\beta_{1}}\cdots t_{n}^{\beta_{n}}h_{w}\:|\:\alpha_{1},\ldots,\alpha_{n}\in \mathbb{Z}, 0\leq\beta_{1},\ldots,\beta_{n}\leq r-1\text{ and }w\in \mathfrak{S}_{n}\big\}.
\end{equation}

Combining these facts with Lemma \ref{isom-YeckeHeck-alge-ac}, we can easily get the following result, which gives a new presentation of $\widehat{Y}_{r,n}$.
\begin{theorem}\label{isom-YeckeHeck-alge-ac-add}
Let $\widehat{H}_{r,n}^{\mathcal{R}}=\mathcal{R}\otimes_{\mathcal{A}} \widehat{H}_{r,n}.$ Then we have an $\mathcal{R}$-algebra isomorphism $\widehat{Y}_{r,n}\overset{\sim}{\longrightarrow}\widehat{H}_{r,n}^{\mathcal{R}}.$ Moreover, the following set
\begin{equation}\label{pbwbasis-ac-hrn-add}
\big\{X_{1}^{\alpha_{1}}\cdots X_{n}^{\alpha_{n}}1_{\lambda}g_{w}\:|\:\alpha_{1},\ldots,\alpha_{n}\in \mathbb{Z}, \lambda\in \underline{\mathfrak{s}}_{n}\text{ and }w\in \mathfrak{S}_{n}\big\}
\end{equation}
forms an $\mathcal{A}$-basis of $\widehat{H}_{r,n}.$
\end{theorem}

\section{A construction of the algebra $\widehat{E}_{r,n}$}

In this section, we are largely inspired by the work of Lusztig in [L6] to give the construction of an algebra $\widehat{E}_{r,n}$, which is in fact isomorphic to a direct sum of matrix algebras with coefficients in tensor products of extended affine Hecke algebras of type $A$. Furthermore, we prove that $\widehat{E}_{r,n}$ satisfies these axiomatic properties $P_{1}$, $P_{2}$, $P_{3}$ and $P_{4}$ described in Section 2.2. Thus, it is an affine cellular algebra.

\subsection{A subalgebra of $\widehat{H}_{r,n}$}
We denote by $\underline{\mathfrak{s}}_{n}'$ a set of representatives for the $\mathfrak{S}_{n}$-orbits in $\underline{\mathfrak{s}}_{n}$ such that if $\lambda=(\lambda_{1},\ldots,\lambda_{n})\in \underline{\mathfrak{s}}_{n}',$ we must have $\lambda_{1}=\cdots=\lambda_{n_1}=\zeta_{1},$ $\lambda_{n_1+1}=\cdots=\lambda_{n_1+n_2}=\zeta_{2},\ldots,$ $\lambda_{n_1+\cdots+n_{r-1}+1}=\cdots=\lambda_{n_1+\cdots+n_{r-1}+n_{r}}=\zeta_{r}$ for some (uniquely determined) non-negative integers $n_1,n_2,\ldots,n_{r}$ satisfying $\sum_{1\leq i\leq r}n_{i}=n.$ In this subsection we fix such $\lambda\in \underline{\mathfrak{s}}_{n}'$ and its associated non-negative integers $n_1,n_2,\ldots,n_{r},$ and consider the algebra $1_{\lambda}\widehat{H}_{r,n}1_{\lambda}.$

Set $W: =\mathfrak{S}_{n}$ and $\widehat{W} :=\mathbb{Z}^{n}\rtimes W.$ It is well-known that $\widehat{W}$ is isomorphic to the \emph{extended affine Weyl group of type} $A,$ and we will not distinguish between them in the following. Let $\preceq$ be the \emph{Bruhat order} on $\widehat{W}.$ Let $\widehat{H}_{n}$ denote the \emph{extended affine Hecke algebra of type} $A$ associated to $\widehat{W},$ which is defined over $\mathcal{A}.$ Here we omit the definition of $\widehat{H}_{n}.$ It is well-known that $\widehat{H}_{n}$ has a Bernstein-Lusztig presentation and an Iwahori-Matsumoto presentation. Accordingly, $\widehat{H}_{n}$ has two different bases, which we will assume to be $\{Z_{1}^{a_{1}}\cdots Z_{n}^{a_{n}}T_{w}\:|\:a_{1},\ldots,a_{n}\in \mathbb{Z}\text{ and }w\in \mathfrak{S}_{n}\}$ and $\{T_{\widehat{w}}\:|\:\widehat{w}\in \widehat{W}\},$ respectively. Besides, $\widehat{H}_{n}$ has a third basis, that is, a \emph{Kazhdan-Lusztig basis} $\{c_{\widehat{w}}\:|\:\widehat{w}\in \widehat{W}\}.$

Recall that $W_{\lambda}=\{w\in \mathfrak{S}_{n}\:|\:w(\lambda)=\lambda\}.$ Since $\lambda\in \underline{\mathfrak{s}}_{n}',$ $W_{\lambda}$ is the \emph{Young subgroup} of $W,$ that is, $W_{\lambda}=\mathfrak{S}_{n_1}\times\mathfrak{S}_{n_2}\times\cdots\times\mathfrak{S}_{n_{r}}.$ Let $\{s_{i}\}$ be a set of generators of the Coxeter group $W_{\lambda}.$ Set $\widehat{W}_{\lambda} :=\mathbb{Z}^{n}\rtimes W_{\lambda}.$ It is well-known that $\widehat{W}_{\lambda}$ is isomorphic to a direct product of some extended affine Weyl groups of type $A,$ and we will not distinguish between them in the following. Since $\widehat{W}_{\lambda}$ is a subgroup of $\widehat{W},$ it inherits the Bruhat order $\preceq$ from $\widehat{W}.$

Set $\widehat{\mathfrak{S}}_{n_{i}} :=\mathbb{Z}^{n_{i}}\rtimes \mathfrak{S}_{n_{i}}$ for $1\leq i\leq r.$ Let $\widehat{H}_{n_{i}}$ denote the extended affine Hecke algebra of type $A$ associated to $\widehat{\mathfrak{S}}_{n_{i}},$ which is defined over $\mathcal{A}.$ We denote by $\widehat{H}_{\lambda}$ the $\mathcal{A}$-subalgebra of $\widehat{H}_{n}$ generated by the elements $Z_{1}^{a_{1}}\cdots Z_{n}^{a_{n}}T_{w}$ (with $a_{1},\ldots,a_{n}\in \mathbb{Z}\text{ and }w\in W_{\lambda}$), which coincides with the $\mathcal{A}$-subalgebra of $\widehat{H}_{n}$ generated by the elements $T_{\widehat{w}}$ (with $\widehat{w}\in \widehat{W}_{\lambda}$). We see that $\widehat{H}_{\lambda}$ is naturally isomorphic to $\widehat{H}_{n_{1}}\otimes \cdots\otimes \widehat{H}_{n_{r}},$ where the tensor products are over $\mathcal{A}.$ Moreover, $\widehat{H}_{\lambda}$ has a Kazhdan-Lusztig basis $\{c_{\widehat{w}}\:|\:\widehat{w}\in \widehat{W}_{\lambda}\}.$
\begin{lemma}\label{isomorphism-theorem-ac}
There exists an isomorphism of $\mathcal{A}$-algebras $\varphi: \widehat{H}_{\lambda}\longrightarrow 1_{\lambda}\widehat{H}_{r,n}1_{\lambda},$ which is given by
\begin{align*}
T_{s_{i}}\longmapsto g_{s_{i}}1_{\lambda}~\text{ for }s_{i}\in W_{\lambda},\quad Z_{j}^{\pm 1}\longmapsto X_{j}^{\pm 1}1_{\lambda}~\text{ for }1\leq j\leq n,\quad\text{and}\quad 1\longmapsto 1_{\lambda}.
\end{align*}
\end{lemma}
\begin{proof}
The map $\varphi$ is obviously a bijective $\mathcal{A}$-linear map. It suffices to show that $\varphi$ is an $\mathcal{A}$-algebra homomorphism, which can be easily verified. We leave it as an exercise to the reader.
\end{proof}

For each $\widehat{w}\in \widehat{W}_{\lambda},$ we denote by $g_{\widehat{w}, \lambda}$ the image of the generator $T_{\widehat{w}}$ of $\widehat{H}_{\lambda}$ under $\varphi$ in Lemma \ref{isomorphism-theorem-ac}. Then $\{g_{\widehat{w}, \lambda}\}_{\widehat{w}\in \widehat{W}_{\lambda}}$ is an $\mathcal{A}$-basis of $1_{\lambda}\widehat{H}_{r,n}1_{\lambda}.$ Recall that there exists a unique ring involution $\bm{\overline{\hspace{0.075cm}\cdot\hspace{0.075cm}}}$ on $\widehat{H}_{\lambda}$ such that $\overline{T_{\widehat{w}}}=T_{\widehat{w}^{-1}}^{-1}$ for each $\widehat{w}\in \widehat{W}_{\lambda}$ and $\overline{q^{m}}=q^{-m}.$ This induces an involution $\bm{\overline{\hspace{0.075cm}\cdot\hspace{0.075cm}}}$ on $1_{\lambda}\widehat{H}_{r,n}1_{\lambda}$ such that $\overline{\varphi(h)}=\varphi(\overline{h})$ for any $h\in \widehat{H}_{\lambda}.$

From the above definition, for any $\widehat{w}\in \widehat{W}_{\lambda}$ we have
\begin{align*}
\overline{g_{\widehat{w}, \lambda}}-g_{\widehat{w}, \lambda}\in \sum_{\widehat{y}\in \widehat{W}_{\lambda}; \widehat{y}\prec \widehat{w}}\mathcal{A}g_{\widehat{y}, \lambda}.
\end{align*}
By an argument similar to the one in [L5, $\S5.2$], we see that for any $\widehat{w}\in \widehat{W}_{\lambda},$ there exists a unique element $c_{\widehat{w}, \lambda}\in 1_{\lambda}\widehat{H}_{r,n}1_{\lambda}$ such that $\overline{c_{\widehat{w}, \lambda}}=c_{\widehat{w}, \lambda},$ and
\begin{align*}
c_{\widehat{w}, \lambda}-g_{\widehat{w}, \lambda}\in \sum_{\widehat{y}\in \widehat{W}_{\lambda}; \widehat{y}\neq \widehat{w}}q^{-1}\mathbb{Z}[q^{-1}]g_{\widehat{y}, \lambda}.
\end{align*}
In fact, the set $\{c_{\widehat{w}, \lambda}\}_{\widehat{w}\in \widehat{W}_{\lambda}}$ is also an $\mathcal{A}$-basis of $1_{\lambda}\widehat{H}_{r,n}1_{\lambda}.$ Moreover, we have $c_{\widehat{w}, \lambda}=\varphi(c_{\widehat{w}})$ for any $\widehat{w}\in \widehat{W}_{\lambda}.$

\subsection{A construction of the algebra $\widehat{E}_{r,n}$}
For each $\lambda\in \underline{\mathfrak{s}}_{n}$ we define $\lambda^{0}\in \underline{\mathfrak{s}}_{n}'$ by $\lambda^{0}\in W\lambda.$ We set
\begin{align*}
\Gamma :=\big\{(\lambda_{1}, \lambda_{2})\in \underline{\mathfrak{s}}_{n}\times \underline{\mathfrak{s}}_{n}\:|\:W\lambda_{1}=W\lambda_{2}\big\}.
\end{align*}

Let $\widehat{E}_{r,n}$ be the set of all formal sums $x=\sum_{(\lambda_{1}, \lambda_{2})\in \Gamma}x_{\lambda_{1}, \lambda_{2}},$ where $x_{\lambda_{1}, \lambda_{2}}\in 1_{\lambda_{1}^{0}}\widehat{H}_{r,n}1_{\lambda_{2}^{0}}.$ Then $\widehat{E}_{r,n}$ is naturally an $\mathcal{A}$-algebra, where the product $xy$ of two elements $x,y\in \widehat{E}_{r,n}$ is given by $(xy)_{\lambda_{1}, \lambda_{2}}=\sum_{\tilde{\lambda}\in W\lambda_{1}}x_{\lambda_{1}, \tilde{\lambda}}y_{\tilde{\lambda}, \lambda_{2}}.$ This algebra $\widehat{E}_{r,n}$ has a unit element, namely the element $1$ such that $1_{\lambda_{1}, \lambda_{2}}=\delta_{\lambda_{1}, \lambda_{2}}1_{\lambda_{1}}$ for any $(\lambda_{1}, \lambda_{2})\in \Gamma.$ Note that for each $(\lambda_{1}, \lambda_{2})\in \Gamma,$ we have $\lambda_{1}^{0}=\lambda_{2}^{0}.$ Thus, we can define a ring involution $\bm{\overline{\hspace{0.075cm}\cdot\hspace{0.075cm}}}: \widehat{E}_{r,n}\rightarrow \widehat{E}_{r,n}$ by $x\mapsto \overline{x}$, where $\overline{x}_{\lambda_{1}, \lambda_{2}}=\overline{x_{\lambda_{1}, \lambda_{2}}}$ for any $(\lambda_{1}, \lambda_{2})\in \Gamma.$

Recall that each element $\widehat{w}$ of $\widehat{W}$ can be uniquely written as the form $\lambda w$, where $\lambda\in \mathbb{Z}^{n}$ and $w\in W$. In what follows, for each element $\widehat{w}\in \widehat{W}$, we will denote by $w$ the corresponding element belonging to $W$ under the above decomposition of $\widehat{w}$.

Let $C=\{(\lambda_{1}, \lambda_{2}, \widehat{w})\in \underline{\mathfrak{s}}_{n}\times \underline{\mathfrak{s}}_{n}\times \widehat{W}\:|\:w\lambda_{1}^{0}=\lambda_{1}^{0}=\lambda_{2}^{0}\}.$ For each $(\lambda_{1}, \lambda_{2}, \widehat{w})\in C,$ we define $x^{\lambda_{1}, \lambda_{2}; \widehat{w}}\in \widehat{E}_{r,n}$ by
\begin{align*}
x^{\lambda_{1}, \lambda_{2}; \widehat{w}}_{\lambda_{1}', \lambda_{2}'}=\delta_{(\lambda_{1}', \lambda_{2}'), (\lambda_{1}, \lambda_{2})}g_{\widehat{w}, \lambda_{1}^{0}}.
\end{align*}
Then $\{x^{\lambda_{1}, \lambda_{2}; \widehat{w}}\:|\:(\lambda_{1}, \lambda_{2}, \widehat{w})\in C\}$ forms an $\mathcal{A}$-basis of $\widehat{E}_{r,n}.$ From the definition, for each $(\lambda_{1}, \lambda_{2}, \widehat{w})\in C$ we have
\begin{align*}
\overline{x^{\lambda_{1}, \lambda_{2}; \widehat{w}}}-x^{\lambda_{1}, \lambda_{2}; \widehat{w}}\in \sum_{\widehat{y}\in \widehat{W}_{\lambda_{1}^{0}}; \widehat{y}\prec \widehat{w}}\mathcal{A}x^{\lambda_{1}, \lambda_{2}; \widehat{y}}.
\end{align*}
By an argument similar to the one in [L5, $\S5.2$], we see that for any $(\lambda_{1}, \lambda_{2}, \widehat{w})\in C,$ there exists a unique element $c^{\lambda_{1}, \lambda_{2}; \widehat{w}}\in \widehat{E}_{r,n}$ such that $\overline{c^{\lambda_{1}, \lambda_{2}; \widehat{w}}}=c^{\lambda_{1}, \lambda_{2}; \widehat{w}},$ and
\begin{align*}
c^{\lambda_{1}, \lambda_{2}; \widehat{w}}-x^{\lambda_{1}, \lambda_{2}; \widehat{w}}\in \sum_{\widehat{y}\in \widehat{W}_{\lambda_{1}^{0}}; \widehat{y}\neq \widehat{w}}q^{-1}\mathbb{Z}[q^{-1}]x^{\lambda_{1}, \lambda_{2}; \widehat{y}}.
\end{align*}
In fact, the set $\{c^{\lambda_{1}, \lambda_{2}; \widehat{w}}\:|\:(\lambda_{1}, \lambda_{2}, \widehat{w})\in C\}$ is also an $\mathcal{A}$-basis of $\widehat{E}_{r,n}.$ Moreover, there exists a unique $\mathcal{A}$-involution $i$ on $\widehat{E}_{r,n}$ such that $i(c^{\lambda_{1}, \lambda_{2}; \widehat{w}})=c^{\lambda_{2}, \lambda_{1}; \widehat{w}^{-1}}.$
\begin{theorem}\label{ern-properties}
The $\mathcal{A}$-algebra $\widehat{E}_{r,n}$ with the basis $\{c^{\lambda_{1}, \lambda_{2}; \widehat{w}}\:|\:(\lambda_{1}, \lambda_{2}, \widehat{w})\in C\}$ and with the $\mathcal{A}$-involution $i$ satisfies properties $P_{1}$, $P_{2}$, $P_{3}$ and $P_{4}$ presented in Section 2.2.
\end{theorem}
\begin{proof}
For each $\lambda\in \underline{\mathfrak{s}}_{n}',$ let $n_{\lambda}=|W\lambda|$ and let Mat$_{n_{\lambda}}(1_{\lambda}\widehat{H}_{r,n}1_{\lambda})$ be the algebra of $n_{\lambda}\times n_{\lambda}$ matrices with coefficients in $1_{\lambda}\widehat{H}_{r,n}1_{\lambda}.$ From the definition of $\widehat{E}_{r,n}$, we have the following decomposition:
\begin{align}\label{the-canonical-decomposition}
\widehat{E}_{r,n}=\bigoplus_{\lambda\in \underline{\mathfrak{s}}_{n}'}\mathrm{Mat}_{n_{\lambda}}(1_{\lambda}\widehat{H}_{r,n}1_{\lambda}),
\end{align}
which is compatible with the algebra structures and with the natural bases. Using this fact, we can easily get that this theorem is reduced to the similar statement for each algebra $1_{\lambda}\widehat{H}_{r,n}1_{\lambda}$ (with $\lambda\in \underline{\mathfrak{s}}_{n}'$). By Lemma \ref{isomorphism-theorem-ac}, it suffices to verify the properties for each algebra $\widehat{H}_{\lambda}$ (with $\lambda\in \underline{\mathfrak{s}}_{n}'$).

Recall that associated with each $\lambda\in \underline{\mathfrak{s}}_{n}',$ we have some non-negative integers $n_1,n_2,\ldots,$ $n_{r}$ such that $\widehat{H}_{\lambda}\cong\widehat{H}_{n_{1}}\otimes \cdots\otimes \widehat{H}_{n_{r}},$ where each $\widehat{H}_{n_{i}}$ ($1\leq i\leq r$) is the extended affine Hecke algebra of type $A.$ It is easy to see that in order to show that $\widehat{H}_{\lambda}$ satisfies the properties, it suffices to verify the properties for each algebra $\widehat{H}_{n_{i}}$ ($1\leq i\leq r$). It follows from [L1-3, Sh] that the algebra $\widehat{H}_{n_{i}}$ satisfies properties $P_{1}$, $P_{2}$ and $P_{3}.$ The fact that the algebra $\widehat{H}_{n_{i}}$ satisfies the property $P_{4}$ follows from [Xi1]. We are done.
\end{proof}

\section{Affine cellularity of affine Yokonuma-Hecke algebras}

Lusztig [L6, $\S34.10$] proved that the unipotent Hecke algebra is isomorphic to a direct sum of matrix algebras. In particular, it can be deduced that the Yokonuma-Hecke algebra is isomorphic to a direct sum of matrix algebras with coefficients in tensor products of Hecke algebras of type A (see also [JPA]). In this section we shall follow Lusztig's approach to prove that the algebra $\widehat{H}_{r,n}$ is isomorphic to the algebra $\widehat{E}_{r,n},$ which are constructed in Sections 3 and 4, respectively. As a consequence, we can show that $\widehat{H}_{r,n},$ and further $\widehat{Y}_{r,n},$ is an affine cellular algebra.

\begin{theorem}\label{iso-theo-ern-yrn-ac}
We have an $\mathcal{A}$-algebra isomorphism $\widehat{H}_{r,n}\overset{\sim}{\longrightarrow}\widehat{E}_{r,n}.$
\end{theorem}
\begin{proof}
For each $\lambda\in \underline{\mathfrak{s}}_{n},$ we choose a sequence $\mathbf{s}_{\lambda}=(s_{1}, s_{2},\ldots, s_{k})$, where each $s_{i}$ is a simple reflection in $\mathfrak{S}_{n},$ and $\lambda^{0}=s_{1}s_{2}\cdots s_{k}\lambda\neq s_{2}\cdots s_{k}\lambda\neq\cdots\neq s_{k}\lambda\neq \lambda,$ or equivalently, $\lambda=s_{k}\cdots s_{2}s_{1}\lambda^{0}\neq s_{k-1}\cdots s_{1}\lambda^{0}\neq\cdots \neq s_{1}\lambda^{0}\neq \lambda^{0}.$ We set
\begin{align*}
\tau_{\lambda} :=g_{s_{1}}g_{s_{2}}\cdots g_{s_{k}}\in \widehat{H}_{r,n}\text{ and }\tau'_{\lambda} :=g_{s_{k}}\cdots g_{s_{2}}g_{s_{1}}\in \widehat{H}_{r,n}.
\end{align*}

We show that
\begin{align}\label{align-alg-isomor-511-hrn-ern}
1_{\lambda^{0}}\tau_{\lambda}\tau_{\lambda}'=1_{\lambda^{0}}\quad \mathrm{and}\quad 1_{\lambda}\tau_{\lambda}'\tau_{\lambda}=1_{\lambda}.
\end{align}
By definition we have $\tau_{\lambda}1_{\lambda}=1_{\lambda^{0}}\tau_{\lambda}$ and $1_{\lambda}\tau_{\lambda}'=\tau_{\lambda}'1_{\lambda^{0}}.$ Thus, we get
\begin{align*}
1_{\lambda^{0}}\tau_{\lambda}\tau_{\lambda}'=\tau_{\lambda}1_{\lambda}\tau_{\lambda}'=g_{s_{1}}g_{s_{2}}\cdots g_{s_{k}}1_{\lambda}g_{s_{k}}\cdots g_{s_{2}}g_{s_{1}}.
\end{align*}
Since $s_{k}\lambda\neq \lambda,$ we can replace $g_{s_{k}}1_{\lambda}g_{s_{k}}$ with $1_{s_{k}\lambda}$ by \eqref{rel-def-Y1-ac-another}, and we obtain
\begin{align*}
1_{\lambda^{0}}\tau_{\lambda}\tau_{\lambda}'=g_{s_{1}}g_{s_{2}}\cdots g_{s_{k-1}}1_{s_{k}\lambda}g_{s_{k-1}}\cdots g_{s_{2}}g_{s_{1}}.
\end{align*}
Since $s_{k-1}s_{k}\lambda\neq s_{k}\lambda,$ we can replace $g_{s_{k-1}}1_{s_{k}\lambda}g_{s_{k-1}}$ with $1_{s_{k-1}s_{k}\lambda}$ by \eqref{rel-def-Y1-ac-another}, and we obtain
\begin{align*}
1_{\lambda^{0}}\tau_{\lambda}\tau_{\lambda}'=g_{s_{1}}g_{s_{2}}\cdots g_{s_{k-2}}1_{s_{k-1}s_{k}\lambda}g_{s_{k-2}}\cdots g_{s_{2}}g_{s_{1}}.
\end{align*}
Continuing in this way we get $$1_{\lambda^{0}}\tau_{\lambda}\tau_{\lambda}'=1_{s_{1}s_{2}\cdots s_{k}\lambda}=1_{\lambda^{0}}.$$ This proves the first identity in \eqref{align-alg-isomor-511-hrn-ern}. The second identity in \eqref{align-alg-isomor-511-hrn-ern} can be proved similarly.

We define an $\mathcal{A}$-linear map $\Psi: \widehat{H}_{r,n}\rightarrow \widehat{E}_{r,n}$ by
\begin{align*}
\Psi(h)_{\lambda_{1}, \lambda_{2}}=\tau_{\lambda_{1}}1_{\lambda_{1}}h1_{\lambda_{2}}\tau_{\lambda_{2}}'\in 1_{\lambda_{1}^{0}}\widehat{H}_{r,n}1_{\lambda_{2}^{0}}\quad \text{ for any }(\lambda_{1}, \lambda_{2})\in \Gamma.
\end{align*}
We show that $\Psi$ is a ring homomorphism. Let $h, h'\in \widehat{H}_{r,n}.$ We have
\begin{align*}
(\Psi(h)\Psi(h'))_{\lambda_{1}, \lambda_{2}}&=\sum_{\tilde{\lambda}\in W\lambda_{1}}\Psi(h)_{\lambda_{1}, \tilde{\lambda}}\Psi(h')_{\tilde{\lambda}, \lambda_{2}}\\
&=\sum_{\tilde{\lambda}\in W\lambda_{1}}\tau_{\lambda_{1}}1_{\lambda_{1}}h1_{\tilde{\lambda}}\tau_{\tilde{\lambda}}'
\tau_{\tilde{\lambda}}1_{\tilde{\lambda}}h'1_{\lambda_{2}}\tau_{\lambda_{2}}'\\
&=\sum_{\tilde{\lambda}\in W\lambda_{1}}\tau_{\lambda_{1}}1_{\lambda_{1}}h1_{\tilde{\lambda}}h'1_{\lambda_{2}}\tau_{\lambda_{2}}',
\end{align*}
where the last equality follows from \eqref{align-alg-isomor-511-hrn-ern}. Since $1_{\lambda_{1}}h1_{\tilde{\lambda}}=0$ if $\tilde{\lambda}\in \underline{\mathfrak{s}}_{n}-W\lambda_{1},$ we have
\begin{align*}
(\Psi(h)\Psi(h'))_{\lambda_{1}, \lambda_{2}}&=\tau_{\lambda_{1}}1_{\lambda_{1}}h\Big(\sum_{\tilde{\lambda}\in \underline{\mathfrak{s}}_{n}}1_{\tilde{\lambda}}\Big)h'1_{\lambda_{2}}\tau_{\lambda_{2}}'\\
&=\tau_{\lambda_{1}}1_{\lambda_{1}}hh'1_{\lambda_{2}}\tau_{\lambda_{2}}'\\
&=\Psi(hh')_{\lambda_{1}, \lambda_{2}}.
\end{align*}
Thus, we have $\Psi(h)\Psi(h')=\Psi(hh')$ as required.

We also define an $\mathcal{A}$-linear map $\Phi: \widehat{E}_{r,n}\rightarrow \widehat{H}_{r,n}$ by
\begin{align*}
\Phi\Big(\sum_{(\lambda_{1}, \lambda_{2})\in \Gamma}x_{\lambda_{1}, \lambda_{2}}\Big)=\sum_{(\lambda_{1}, \lambda_{2})\in \Gamma}\tau_{\lambda_{1}}'1_{\lambda_{1}^{0}}x_{\lambda_{1}, \lambda_{2}}1_{\lambda_{2}^{0}}\tau_{\lambda_{2}}.
\end{align*}
We show that $\Phi$ is a ring homomorphism. Let $x,$ $y\in \widehat{E}_{r,n}.$ We have
\begin{align*}
\Phi(xy)&=\Phi\Big(\sum_{(\lambda_{1}, \lambda_{2})\in \Gamma}(xy)_{\lambda_{1}, \lambda_{2}}\Big)\\
&=\sum_{(\lambda_{1}, \lambda_{2})\in \Gamma}\tau_{\lambda_{1}}'1_{\lambda_{1}^{0}}(xy)_{\lambda_{1}, \lambda_{2}}1_{\lambda_{2}^{0}}\tau_{\lambda_{2}}\\
&=\sum_{(\lambda_{1}, \lambda_{2})\in \Gamma}\tau_{\lambda_{1}}'1_{\lambda_{1}^{0}}\Big(\sum_{\tilde{\lambda}\in W\lambda_{1}}x_{\lambda_{1}, \tilde{\lambda}}y_{\tilde{\lambda}, \lambda_{2}}\Big)1_{\lambda_{2}^{0}}\tau_{\lambda_{2}},
\end{align*}
and
\begin{align*}
\Phi(x)\Phi(y)&=\Big(\sum_{(\lambda_{1}, \lambda_{2})\in \Gamma}\tau_{\lambda_{1}}'1_{\lambda_{1}^{0}}x_{\lambda_{1}, \lambda_{2}}1_{\lambda_{2}^{0}}\tau_{\lambda_{2}}\Big)\Big(\sum_{(\lambda_{1}', \lambda_{2}')\in \Gamma}\tau_{\lambda_{1}'}'1_{\lambda_{1}'^{0}}y_{\lambda_{1}', \lambda_{2}'}1_{\lambda_{2}'^{0}}\tau_{\lambda_{2}'}\Big)\\
&=\sum_{(\lambda_{1}, \lambda_{2}')\in \Gamma}\tau_{\lambda_{1}}'1_{\lambda_{1}^{0}}\Big(\sum_{\lambda_{2}\in W\lambda_{1}}x_{\lambda_{1}, \lambda_{2}}1_{\lambda_{2}^{0}}\tau_{\lambda_{2}}\tau_{\lambda_{2}}'1_{\lambda_{2}^{0}}y_{\lambda_{2}, \lambda_{2}'}\Big)1_{\lambda_{2}'^{0}}\tau_{\lambda_{2}'}\\
&=\sum_{(\lambda_{1}, \lambda_{2}')\in \Gamma}\tau_{\lambda_{1}}'1_{\lambda_{1}^{0}}\Big(\sum_{\lambda_{2}\in W\lambda_{1}}x_{\lambda_{1}, \lambda_{2}}y_{\lambda_{2}, \lambda_{2}'}\Big)1_{\lambda_{2}'^{0}}\tau_{\lambda_{2}'}.
\end{align*}
Thus, we have $\Phi(xy)=\Phi(x)\Phi(y)$ as required.

Next we want to show that $\Phi\circ\Psi(h)=h$ for any $h\in \widehat{H}_{r,n}$ and $\Psi\circ\Phi(x)=x$ for any $x\in \widehat{E}_{r,n}.$

We have
\begin{align*}
\Phi\circ\Psi(h)&=\Phi\Big(\sum_{(\lambda_{1}, \lambda_{2})\in \Gamma}\Psi(h)_{\lambda_{1}, \lambda_{2}}\Big)\\
&=\sum_{(\lambda_{1}, \lambda_{2})\in \Gamma}\tau_{\lambda_{1}}'1_{\lambda_{1}^{0}}\Psi(h)_{\lambda_{1}, \lambda_{2}}1_{\lambda_{2}^{0}}\tau_{\lambda_{2}}\\
&=\sum_{(\lambda_{1}, \lambda_{2})\in \Gamma}\tau_{\lambda_{1}}'1_{\lambda_{1}^{0}}\tau_{\lambda_{1}}1_{\lambda_{1}}h1_{\lambda_{2}}\tau_{\lambda_{2}}'1_{\lambda_{2}^{0}}\tau_{\lambda_{2}}\\
&=\sum_{(\lambda_{1}, \lambda_{2})\in \Gamma}1_{\lambda_{1}}h1_{\lambda_{2}}\\
&=h,
\end{align*}
where the last equality follows from the fact that $1_{\lambda_{1}}h1_{\lambda_{2}}=0$ whenever $(\lambda_{1}, \lambda_{2})\notin \Gamma.$

We have
\begin{align*}
\Psi\circ\Phi(x)&=\Psi\Big(\sum_{(\lambda_{1}, \lambda_{2})\in \Gamma}\tau_{\lambda_{1}}'1_{\lambda_{1}^{0}}x_{\lambda_{1}, \lambda_{2}}1_{\lambda_{2}^{0}}\tau_{\lambda_{2}}\Big)\\
&=\sum_{(\lambda_{1}', \lambda_{2}')\in \Gamma}\tau_{\lambda_{1}'}1_{\lambda_{1}'}\Big(\sum_{(\lambda_{1}, \lambda_{2})\in \Gamma}\tau_{\lambda_{1}}'1_{\lambda_{1}^{0}}x_{\lambda_{1}, \lambda_{2}}1_{\lambda_{2}^{0}}\tau_{\lambda_{2}}\Big)1_{\lambda_{2}'}\tau_{\lambda_{2}'}'\\
&=\sum_{(\lambda_{1}, \lambda_{2})\in \Gamma}\tau_{\lambda_{1}}1_{\lambda_{1}}\tau_{\lambda_{1}}'1_{\lambda_{1}^{0}}x_{\lambda_{1}, \lambda_{2}}1_{\lambda_{2}^{0}}\tau_{\lambda_{2}}1_{\lambda_{2}}\tau_{\lambda_{2}}'\\
&=\sum_{(\lambda_{1}, \lambda_{2})\in \Gamma}1_{\lambda_{1}^{0}}x_{\lambda_{1}, \lambda_{2}}1_{\lambda_{2}^{0}}\\
&=x.
\end{align*}

We have proved this theorem.
\end{proof}

Recall that $\mathcal{R}=\mathbb{Z}[\frac{1}{r}][q,q^{-1},\zeta].$ Set $\widehat{E}_{r,n}^{\mathcal{R}} :=\mathcal{R}\otimes_{\mathcal{A}} \widehat{E}_{r,n}$ and $\widehat{H}_{\lambda}^{\mathcal{R}} :=\mathcal{R}\otimes_{\mathcal{A}} \widehat{H}_{\lambda}$ for each $\lambda\in \underline{\mathfrak{s}}_{n}'.$ Combining Theorems \ref{isom-YeckeHeck-alge-ac-add} and \ref{iso-theo-ern-yrn-ac} with Lemma \ref{isomorphism-theorem-ac} and the equality \eqref{the-canonical-decomposition}, we can get the following theorem.
\begin{theorem}\label{isomorphism-the-tensorproduct-hecke}
We have the following isomorphism of $\mathcal{R}$-algebras$:$
$$\widehat{Y}_{r,n}\cong\widehat{E}_{r,n}^{\mathcal{R}}\cong \bigoplus_{\lambda\in \underline{\mathfrak{s}}_{n}'}\mathrm{Mat}_{n_{\lambda}}(\widehat{H}_{\lambda}^{\mathcal{R}}).$$
\end{theorem}

Combining Theorems \ref{affine-cell-theorem} and \ref{ern-properties} with Theorem \ref{iso-theo-ern-yrn-ac}, we have proved the following result.
\begin{theorem}\label{hrn-affine-cellularity}
Let $\mathcal{A}=\mathbb{Z}[q, q^{-1}].$ The algebra $\widehat{H}_{r,n}$ over $\mathcal{A}$ is an affine cellular algebra.
\end{theorem}

Combining Theorems \ref{isom-YeckeHeck-alge-ac-add} and \ref{hrn-affine-cellularity} with Lemma \ref{affine-cellular-algebra-field-extension}, we can get the following result.
\begin{theorem}\label{yrn-affine-cellularity}
Let $\mathcal{R}=\mathbb{Z}[\frac{1}{r}][q,q^{-1},\zeta].$ The affine Yokonuma-Hecke algebra $\widehat{Y}_{r,n}$ over $\mathcal{R}$ is an affine cellular algebra.
\end{theorem}

\noindent \begin{remark} By Theorem \ref{isomorphism-the-tensorproduct-hecke}, we can also recover the modular representation theory of $\widehat{Y}_{r,n}$ established in [CW, Theorem 4.1]. We skip the details and leave them to the reader. \end{remark}

Next we will apply the approach of affine cellular algebras to investigate homological properties of $\widehat{Y}_{r,n}.$ Let $\tau_{\lambda}$, $\tau_{\lambda}'$ $(\lambda\in \underline{\mathfrak{s}}_{n})$ and the map $\Psi$ be defined as in the proof of Theorem \ref{iso-theo-ern-yrn-ac}. Then we have the following lemma.
\begin{lemma}\label{important-formula-ac}
For any $(\lambda_{1}, \lambda_{2}, \widehat{w})\in C,$ we have $\Psi(\tau_{\lambda_{1}}'1_{\lambda_{1}^{0}}c_{\widehat{w}, \lambda_{1}^{0}}1_{\lambda_{2}^{0}}\tau_{\lambda_{2}})=c^{\lambda_{1}, \lambda_{2}; \widehat{w}}.$
\end{lemma}
\begin{proof}
We first show that for any $(\lambda_{1}, \lambda_{2}, \widehat{w})\in C,$ we have
\begin{align}\label{lemma-new-55-lemmapsi}
\Psi(\tau_{\lambda_{1}}'1_{\lambda_{1}^{0}}g_{\widehat{w}, \lambda_{1}^{0}}1_{\lambda_{2}^{0}}\tau_{\lambda_{2}})=x^{\lambda_{1}, \lambda_{2}; \widehat{w}}.
\end{align}
We have, for any $(\lambda_{1}', \lambda_{2}')\in \Gamma,$
\begin{align*}
\Psi(\tau_{\lambda_{1}}'1_{\lambda_{1}^{0}}g_{\widehat{w}, \lambda_{1}^{0}}1_{\lambda_{2}^{0}}\tau_{\lambda_{2}})_{\lambda_{1}', \lambda_{2}'}=\tau_{\lambda_{1}'}1_{\lambda_{1}'}\tau_{\lambda_{1}}'1_{\lambda_{1}^{0}}g_{\widehat{w}, \lambda_{1}^{0}}1_{\lambda_{2}^{0}}\tau_{\lambda_{2}}1_{\lambda_{2}'}\tau_{\lambda_{2}'}'=\delta_{(\lambda_{1}', \lambda_{2}'), (\lambda_{1}, \lambda_{2})}g_{\widehat{w}, \lambda_{1}^{0}}.
\end{align*}
The equality \eqref{lemma-new-55-lemmapsi} now follows from the definition of $x^{\lambda_{1}, \lambda_{2}; \widehat{w}}.$

Next we show that
\begin{align}\label{lemma-new-55-lemmapsi1111}
\overline{\Psi(\tau_{\lambda_{1}}'1_{\lambda_{1}^{0}}c_{\widehat{w}, \lambda_{1}^{0}}1_{\lambda_{2}^{0}}\tau_{\lambda_{2}})}=\Psi(\tau_{\lambda_{1}}'1_{\lambda_{1}^{0}}c_{\widehat{w}, \lambda_{1}^{0}}1_{\lambda_{2}^{0}}\tau_{\lambda_{2}}).
\end{align}
For any $(\lambda_{1}', \lambda_{2}')\in \Gamma,$ we have
\begin{align*}
\overline{\Psi(\tau_{\lambda_{1}}'1_{\lambda_{1}^{0}}c_{\widehat{w}, \lambda_{1}^{0}}1_{\lambda_{2}^{0}}\tau_{\lambda_{2}})}_{\lambda_{1}', \lambda_{2}'}&=\overline{\Psi(\tau_{\lambda_{1}}'1_{\lambda_{1}^{0}}c_{\widehat{w}, \lambda_{1}^{0}}1_{\lambda_{2}^{0}}\tau_{\lambda_{2}})_{\lambda_{1}', \lambda_{2}'}}\\
&=\overline{\delta_{(\lambda_{1}', \lambda_{2}'), (\lambda_{1}, \lambda_{2})}c_{\widehat{w}, \lambda_{1}^{0}}}\\
&=\delta_{(\lambda_{1}', \lambda_{2}'), (\lambda_{1}, \lambda_{2})}c_{\widehat{w}, \lambda_{1}^{0}}\\
&=\Psi(\tau_{\lambda_{1}}'1_{\lambda_{1}^{0}}c_{\widehat{w}, \lambda_{1}^{0}}1_{\lambda_{2}^{0}}\tau_{\lambda_{2}})_{\lambda_{1}', \lambda_{2}'}.
\end{align*}
Thus, we see that the equality \eqref{lemma-new-55-lemmapsi1111} holds by definition.

It follows from \eqref{lemma-new-55-lemmapsi} that
\begin{align}\label{lemma-new-55-lemmapsi2222}
\Psi(\tau_{\lambda_{1}}'1_{\lambda_{1}^{0}}c_{\widehat{w}, \lambda_{1}^{0}}1_{\lambda_{2}^{0}}\tau_{\lambda_{2}})-x^{\lambda_{1}, \lambda_{2}; \widehat{w}}&=\Psi\big(\tau_{\lambda_{1}}'1_{\lambda_{1}^{0}}(c_{\widehat{w}, \lambda_{1}^{0}}-g_{\widehat{w}, \lambda_{1}^{0}})1_{\lambda_{2}^{0}}\tau_{\lambda_{2}}\big)\notag\\
&\in \sum_{\widehat{y}\in \widehat{W}_{\lambda_{1}^{0}}; \widehat{y}\neq \widehat{w}}q^{-1}\mathbb{Z}[q^{-1}]x^{\lambda_{1}, \lambda_{2}; \widehat{y}}.
\end{align}
According to the uniqueness of the element $c^{\lambda_{1}, \lambda_{2}; \widehat{w}}$ satisfying the two properties \eqref{lemma-new-55-lemmapsi1111} and \eqref{lemma-new-55-lemmapsi2222}, we must have $\Psi(\tau_{\lambda_{1}}'1_{\lambda_{1}^{0}}c_{\widehat{w}, \lambda_{1}^{0}}1_{\lambda_{2}^{0}}\tau_{\lambda_{2}})=c^{\lambda_{1}, \lambda_{2}; \widehat{w}}.$
\end{proof}

Let $k$ be a field of characteristic zero, which is considered as an $\mathcal{A}$-algebra by mapping $q$ to an invertible element $q\in k^{*}.$ Set
\begin{align*}
\widehat{E}_{r,n}^{k} :=k\otimes_{\mathcal{A}} \widehat{E}_{r,n},\quad\widehat{H}_{r,n}^{k} :=k\otimes_{\mathcal{A}} \widehat{H}_{r,n}\quad \mathrm{and} \quad\widehat{H}_{\lambda}^{k} :=k\otimes_{\mathcal{A}} \widehat{H}_{\lambda} \text{ for each }\lambda\in \underline{\mathfrak{s}}_{n}'.
\end{align*}

Combining Theorems \ref{affine-cell-theorem} and \ref{ern-properties} with Lemma \ref{affine-cellular-algebra-field-extension}, we see that the algebra $\widehat{E}_{r,n}^{k}$ is an affine cellular algebra. In the next lemma, we will show that all layers in an affine cell chain of $\widehat{E}_{r,n}^{k}$ satisfy the conditions in Theorem \ref{affine-cellua-YH-alg}.
\begin{lemma}\label{idempotent-condition-ern}
Assume that $q\in k^{*}$ and $\sum_{w\in W}q^{l(w)}\neq 0.$ Then the affine $k$-algebra, associated to each layer of $\widehat{E}_{r,n}^{k}$ as in Proposition \ref{kx2-proposition-affine-cell}, has finite global dimension and its Jacobson radical is zero. Moreover, each layer in an affine cell chain of $\widehat{E}_{r,n}^{k}$ is idempotent and is generated by a non-zero idempotent element.
\end{lemma}
\begin{proof}
Let $\lambda\in \underline{\mathfrak{s}}_{n}.$ Recall that associated with $\lambda^{0}\in \underline{\mathfrak{s}}_{n}',$ we have some uniquely determined non-negative integers $n_1,n_2,\ldots,n_{r}$ such that $\widehat{H}_{\lambda^{0}}^{k}\cong\widehat{H}_{n_{1}}^{k}\otimes \cdots\otimes \widehat{H}_{n_{r}}^{k},$ where each $\widehat{H}_{n_{i}}^{k} :=k\otimes_{\mathcal{A}} \widehat{H}_{n_{i}}$ ($1\leq i\leq r$) is the extended affine Hecke algebra of type $A.$ By [KX2, Theorem 5.7] and Lemma \ref{affine-cellular-algebra-field-extension} we see that the algebra $\widehat{H}_{n_{i}}^{k}$ is an affine cellular algebra. Based on the work [Xi1], it has been claimed (see the proof of [KX2, Theorem 5.8]) that the affine $k$-algebra, associated to each layer in an affine cell chain of $\widehat{H}_{n_{i}}^{k}$ as in Proposition \ref{kx2-proposition-affine-cell}, is isomorphic to a ring of Laurent polynomials in finitely many variables of the form
\begin{equation}\label{the-form-laurent-polynomai-ring}
k[X_1,\ldots, X_{s},X_{s+1},X_{s+1}^{-1},\ldots,X_{t},X_{t}^{-1}]
\end{equation}
for some $t\in \mathbb{Z}_{> 0},$ and moreover, it has finite global dimension and its Jacobson radical is zero. Based on the work [Xi2], it has been proved (see the proof of [KX2, Theorem 5.8] again) that each layer in an affine cell chain of $\widehat{H}_{n_{i}}^{k}$ is idempotent and contains a non-zero idempotent element under the assumption.

By Lemma \ref{lemma-tensor-products-aff}, we see that the algebra $\widehat{H}_{\lambda^{0}}^{k},$ which is isomorphic to a tensor product of the algebras $\widehat{H}_{n_{i}}^{k},$ is an affine cellular algebra. By the proof of Lemma \ref{lemma-tensor-products-aff}, we see that the affine $k$-algebra, associated to each layer in an affine cell chain of $\widehat{H}_{\lambda^{0}}^{k},$ is isomorphic to a tensor product of the rings of the form \eqref{the-form-laurent-polynomai-ring}, and therefore it has finite global dimension and its Jacobson radical is zero; moreover, each layer in an affine cell chain of $\widehat{H}_{\lambda^{0}}^{k}$ is idempotent and contains a non-zero idempotent element.

Assume that $\mathfrak{c}$ is a 2-cell of $\{c_{\widehat{w}, \lambda^{0}}\:|\:\widehat{w}\in \widehat{W}_{\lambda^{0}}\}.$ We denote by $\widehat{H}_{\mathfrak{c}}^{k}$ the free $k$-submodule of $1_{\lambda^{0}}\widehat{H}_{r,n}^{k}1_{\lambda^{0}}$ spanned by $\mathfrak{c}.$ Then $\widehat{H}_{\mathfrak{c}}^{k}$ is a layer in an affine cell chain of $1_{\lambda^{0}}\widehat{H}_{r,n}^{k}1_{\lambda^{0}}.$ By Lemma \ref{isomorphism-theorem-ac} and the explanations in the preceding paragraph, we see that the affine $k$-algebra $B_{\mathfrak{c}},$ associated to $\widehat{H}_{\mathfrak{c}}^{k},$ has finite global dimension and its Jacobson radical is zero; moreover, $\widehat{H}_{\mathfrak{c}}^{k}$ is idempotent and contain a non-zero idempotent element.

By Theorem \ref{ern-properties} and its proof, we see that the 2-cells of $\{c^{\lambda_{1}, \lambda_{2}; \widehat{w}}\:|\:(\lambda_{1}, \lambda_{2}, \widehat{w})\in C\}$ are the sets of the form $\{c^{\lambda_{1}, \lambda_{2}; \widehat{w}}\},$ where $\lambda_{1}, \lambda_{2}$ run through $W\delta^{0}$ (with $\delta^{0}\in \underline{\mathfrak{s}}_{n}'$ fixed) and $\widehat{w}$ runs through a subset $X$ of $\widehat{W}_{\delta^{0}}$ such that $\{c_{\widehat{w}, \delta^{0}}\:|\:\widehat{w}\in X\}$ is a 2-cell of $\{c_{\widehat{w}, \delta^{0}}\:|\:\widehat{w}\in \widehat{W}_{\delta^{0}}\}.$ Let $\overline{\mathfrak{c}}$ be the 2-cell of $\{c^{\lambda_{1}, \lambda_{2}; \widehat{w}}\:|\:(\lambda_{1}, \lambda_{2}, \widehat{w})\in C\}$ corresponding to $\mathfrak{c};$ that is, $\overline{\mathfrak{c}}$ consists of the elements $c^{\lambda_{1}, \lambda_{2}; \widehat{z}},$ where $\lambda_{1}, \lambda_{2}$ run through $W\lambda^{0}$ and $\widehat{z}$ runs through a subset $Z$ of $\widehat{W}_{\lambda^{0}}$ such that $\mathfrak{c}=\{c_{\widehat{z}, \lambda^{0}}\:|\:\widehat{z}\in Z\}.$

Let us denote by $\widehat{E}_{\overline{\mathfrak{c}}}^{k}$ the free $k$-submodule of $\widehat{E}_{r,n}^{k}$ spanned by $\overline{\mathfrak{c}}.$ Then $\widehat{E}_{\overline{\mathfrak{c}}}^{k}$ is a layer in an affine cell chain of $\widehat{E}_{r,n}^{k}.$ We need to show that the affine $k$-algebra $B_{\overline{\mathfrak{c}}}$, associated to $\widehat{E}_{\overline{\mathfrak{c}}}^{k},$ has finite global dimension and its Jacobson radical is zero, and moreover, $\widehat{E}_{\overline{\mathfrak{c}}}^{k}$ is idempotent and is generated by a non-zero idempotent element.

By \eqref{the-canonical-decomposition}, we see that the affine $k$-algebra $B_{\overline{\mathfrak{c}}}$ is isomorphic to a matrix algebra over $B_{\mathfrak{c}}.$ By the facts that the global dimension is preserved under Morita equivalence, and that $\mathrm{rad}M_{m}(R)=M_{m}(\mathrm{rad}R)$ for any ring $R$ and $m\in \mathbb{Z}_{> 0},$ we see that $B_{\overline{\mathfrak{c}}}$ has finite global dimension and its Jacobson radical is zero.

Assume that $c^{\lambda_{1}, \lambda_{2}; \widehat{z}}\in \overline{\mathfrak{c}}$ (necessarily $\lambda_{1}^{0}=\lambda_{2}^{0}=\lambda^{0}$). By Lemma \ref{important-formula-ac} we have $c^{\lambda_{1}, \lambda_{2}; \widehat{z}}=\Psi(\tau_{\lambda_{1}}'1_{\lambda_{1}^{0}}c_{\widehat{z}, \lambda_{1}^{0}}1_{\lambda_{2}^{0}}\tau_{\lambda_{2}})$ with $c_{\widehat{z}, \lambda_{1}^{0}}\in \mathfrak{c}.$ Since $\widehat{H}_{\mathfrak{c}}^{k}$ is idempotent, we assume that there exist two elements $h_1, h_2\in\widehat{H}_{\mathfrak{c}}^{k}$ such that $c_{\widehat{z}, \lambda_{1}^{0}}=h_{1}\cdot h_{2}.$ By \eqref{align-alg-isomor-511-hrn-ern} we have
\begin{align*}
c^{\lambda_{1}, \lambda_{2}; \widehat{z}}&=\Psi(\tau_{\lambda_{1}}'1_{\lambda_{1}^{0}}h_{1}\cdot h_{2}1_{\lambda_{2}^{0}}\tau_{\lambda_{2}})\\
&=\Psi(\tau_{\lambda_{1}}'1_{\lambda_{1}^{0}}h_{1}1_{\lambda_{1}^{0}}\tau_{\lambda_{1}}\tau_{\lambda_{1}}'1_{\lambda_{1}^{0}}
h_{2}1_{\lambda_{2}^{0}}\tau_{\lambda_{2}})\\
&=\Psi(\tau_{\lambda_{1}}'1_{\lambda_{1}^{0}}h_{1}1_{\lambda_{1}^{0}}\tau_{\lambda_{1}})\cdot\Psi(\tau_{\lambda_{1}}'1_{\lambda_{1}^{0}}
h_{2}1_{\lambda_{2}^{0}}\tau_{\lambda_{2}}).
\end{align*}
By using Lemma \ref{important-formula-ac} again, we see that both $\Psi(\tau_{\lambda_{1}}'1_{\lambda_{1}^{0}}h_{1}1_{\lambda_{1}^{0}}\tau_{\lambda_{1}})$ and $\Psi(\tau_{\lambda_{1}}'1_{\lambda_{1}^{0}}h_{2}1_{\lambda_{2}^{0}}\tau_{\lambda_{2}})$ belong to $\widehat{E}_{\overline{\mathfrak{c}}}^{k}.$ Thus, we have proved that $\widehat{E}_{\overline{\mathfrak{c}}}^{k}$ is idempotent.

Assume that $e_{\mathfrak{c}}$ is a non-zero idempotent element of $\widehat{H}_{\mathfrak{c}}^{k},$ that is, $e_{\mathfrak{c}}^{2}=e_{\mathfrak{c}}.$ By using \eqref{align-alg-isomor-511-hrn-ern} we have
\begin{align*}
\Psi(\tau_{\lambda}'1_{\lambda^{0}}e_{\mathfrak{c}}1_{\lambda^{0}}\tau_{\lambda})&=\Psi(\tau_{\lambda}'1_{\lambda^{0}}e_{\mathfrak{c}}^{2}1_{\lambda^{0}}\tau_{\lambda})\\
&=\Psi(\tau_{\lambda}'1_{\lambda^{0}}e_{\mathfrak{c}}1_{\lambda^{0}}\tau_{\lambda}\tau_{\lambda}'1_{\lambda^{0}}e_{\mathfrak{c}}1_{\lambda^{0}}\tau_{\lambda})\\
&=\Psi(\tau_{\lambda}'1_{\lambda^{0}}e_{\mathfrak{c}}1_{\lambda^{0}}\tau_{\lambda})\cdot\Psi(\tau_{\lambda}'1_{\lambda^{0}}e_{\mathfrak{c}}1_{\lambda^{0}}\tau_{\lambda})
\end{align*}
By Lemma \ref{important-formula-ac}, we see that $\Psi(\tau_{\lambda}'1_{\lambda^{0}}e_{\mathfrak{c}}1_{\lambda^{0}}\tau_{\lambda})$ belongs to $\widehat{E}_{\overline{\mathfrak{c}}}^{k}$ and is a non-zero idempotent element of $\widehat{E}_{\overline{\mathfrak{c}}}^{k}.$ By applying [KX2, Theorem 4.3(1)], we conclude that $\widehat{E}_{\overline{\mathfrak{c}}}^{k}$ is generated by the idempotent element. We have proved this lemma.
\end{proof}

We further assume that $k$ contains the $r$-th root of unity $\zeta$, which is regarded as an $\mathcal{R}$-algebra. Set $\widehat{Y}_{r,n}^{k} :=k\otimes_{\mathcal{R}} \widehat{Y}_{r,n}.$ By Theorem \ref{isomorphism-the-tensorproduct-hecke}, we have the following algebra isomorphism:
\begin{align*}
\widehat{Y}_{r,n}^{k}\overset{\sim}{\longrightarrow}\widehat{E}_{r,n}^{k}.
\end{align*}
Combining this fact with Lemma \ref{idempotent-condition-ern}, we can immediately get the following result by applying Theorem \ref{affine-cellua-YH-alg}.
\begin{theorem}\label{final-homological-propertis}
Assume that $k$ is a field of characteristic zero which contains $\zeta$, $q\in k^{*}$ and $\sum_{w\in W}q^{l(w)}\neq 0.$ Then all layers in an affine cell chain of $\widehat{Y}_{r,n}^{k}$ correspond to idempotent two-sided ideals, which all have idempotent generators. In particular, the parameter set of simple $\widehat{Y}_{r,n}^{k}$-modules equals the parameter set of simple modules of its asymptotic algebra, and so it is a finite union of affine spaces. Moreover, $\widehat{Y}_{r,n}^{k}$ has finite global dimension and its derived module category admits a stratification whose strata are the derived module categories of the various affine $k$-algebras $B_{l}.$
\end{theorem}

\noindent \begin{remark} From Theorem \ref{final-homological-propertis}, we immediately get that the affine Yokonuma-Hecke algebra $\widehat{Y}_{r,n}(q)$ is an affine quasi-hereditary algebra when the parameter $q$ is not a root of the Poincar\'{e} polynomial. \end{remark}

\noindent{\bf Acknowledgements.}
The author would like to thank Professor G. Lusztig for his comments on a draft version of this paper. The author would also like to thank the referee for many helpful comments and suggestions. The author was partially supported by Fundamental Research Funds of Shandong University (No. 2016GN024), China Postdoctoral Science Foundation funded project (No. 2016M592171) and by National Natural Science Foundation of China (No. 11601273).



\end{document}